\documentclass{article}

\usepackage[T1]{fontenc} 
\usepackage{mathrsfs}
\usepackage{enumerate}
\usepackage{indentfirst}
\usepackage{ulem}
\usepackage[dvips]{graphicx}
\usepackage{titlesec,float}
\usepackage{amsmath}
\usepackage{xcolor}
\usepackage{listings}
\usepackage{appendix}
\usepackage{titlesec,titletoc}
\usepackage{times}
\usepackage{mathtools}
\usepackage[top=30mm,bottom=30mm,left=30mm,right=30mm]{geometry}
\usepackage{amsfonts}
\usepackage{amssymb}
\usepackage{setspace}
\usepackage{accents}
\usepackage{statex}
\usepackage{amssymb}
\usepackage{tikz}

\numberwithin{equation}{section}
\makeatletter
\def\wideubar{\underaccent{{\cc@style\underline{\mskip10mu}}}}
\makeatother

\makeatletter
\def\widebar{\accentset{{\cc@style\underline{\mskip10mu}}}}
\makeatother

\begin{document}
 
\newtheorem{theorem}{Theorem}[section]
\newtheorem{assume}[theorem]{Assumption}
\newtheorem{proposition}[theorem]{Proposition}
\newtheorem{corollary}[theorem]{Corollary}
\newtheorem{lemma}[theorem]{Lemma}
\newtheorem{definition}[theorem]{Definition}
\newtheorem{remark}[theorem]{Remark}
\newtheorem{claim}[theorem]{Claim}
\newcommand{\RNum}[1]{\uppercase\expandafter{\romannumeral #1\relax}}

\newenvironment{proof}{{\noindent \it{Proof}}\ }{\hfill $\square$\par}
\newenvironment{pf}{{\noindent\it Proof of the claim.}\ }{\hfill $\square$\par}

\title{The long-time behavior of solutions of a three-component reaction-diffusion model for the population dynamics of farmers and hunter-gatherers: the different motility case} 

\author{Dongyuan Xiao, Ryunosuke Mori}
\date{}
\maketitle


\begin{abstract}
In this paper, we investigate the spreading properties of solutions of the Aoki-Shida-Shigesada model. This model is a three-component reaction-diffusion system that delineates the geographical expansion of an initially localized population of farmers into a region occupied by hunter-gatherers. By considering the scenario where farmers and hunter-gatherers possess identical motility, Aoki {\it et al.}  previously concluded, through numerical simulations and some formal linearization arguments, that there are four different types of spreading behaviors depending on the parameter values. 
In this paper, we concentrate on the general case for which farmers and hunter-gatherers possess different motility.  By providing more sophisticated estimates, we not only theoretically justify for the spreading speed of the Aoki-Shida-Shigesada model, but also establish sharp estimates for the long-time behaviors of solutions. These estimates enable us to validate all four types of spreading behaviors observed by Aoki {\it et al.}. 

\vspace{10pt}
\hspace{-14pt}\textbf{Key words:\:Farmer and hunter-gatherer model;\:long time behavior;\:spreading speed;\: logarithmic correction}
\end{abstract}

\section{Introduction} 


In 1996, Aoki, Shida and Shigesada \cite{shigesada model} proposed the following three-component reaction-diffusion system to study the process of Neolithic transition from hunter-gatherers to farmers (actually, in \cite{shigesada model}, they only considered the with case with same diffusion rate $D_f=D_c=D_h$, but in the present paper we formulate the problem with different diffusion speeds): 
\begin{equation*}
\left\{
\begin{aligned}
&\partial_tF=D_f\partial_{xx}F+r_fF(1-(F+C)/K),\\
&\partial_tC=D_c\partial_{xx}C+r_cC(1-(F+C)/K)+e(F+C)H,\ \ \ \ \ \ \mbox{in}\ \mathbb{R}.\\
&\partial_tH=D_h\partial_{xx}H+r_hH(1-H/L)-e(F+C)H,
\end{aligned}
\right.
\end{equation*}
The population densities of initial farmers, converted farmers and hunter-gatherers are represented by $F$, $C$ and $H$, respectively.
This model contains nine positive parameters: $D_f$, $D_c$ and $D_h$ are the diffusion coefficients; $r_f$, $r_c$ and $r_h$ are the intrinsic growth rates; $K$ and $L$ are the carrying capacities of farmers and hunter-gatherers; $e$ is the conversion rate of hunter-gatherers to farmers. Note that, in \cite{shigesada model}, $D_h$ is always assumed to be equal to $D_c$ and $D_f$. However, it is more natural to imagine that hunter-gatherers possess a different motility with  farmers.

As shown in \cite{shigesada model}, by a suitable change of variables, the above system is converted to:
\begin{equation}\label{fch-equation}
\left\{
\begin{aligned}
&\partial_tF= \partial_{xx}F+aF(1-F-C),\\
&\partial_tC=d_c \partial_{xx}C+C(1-F-C)+sH(F+C),\\
&\partial_tH=d_h \partial_{xx}H+bH(1-H-g(F+C)),
\end{aligned}
\right.\quad\quad\mbox{in}\ \ \mathbb{R},
\end{equation}
where $d_c=D_c/D_f$, $d_h=D_h/D_f$, $a=r_f/r_c$, $b=r_h/r_c$, $s=eL/r_c$ and $g=eK/r_h$.
We consider the following initial condition
\begin{equation}\label{initial data}
H(0,x)\equiv 1,\ \ C(0,x)\equiv 0,\ \ F(0,x)=F_0(x)\ge0\ \ \ \ (F_0\not\equiv0),
\end{equation}
where $F_0(x)$ is a compactly supported continuous function. Such initial data indicates a localized population of farmers in a region that was originally occupied by hunter-gatherers. 

The ODE system related to the system (\ref{fch-equation}) possesses the following four different types of steady states: 
\[
(F,C,H)=
\left\{
\begin{aligned}
&(0,0,0),\ (0,0,1),\\
&(\widehat{F},\widehat{C},0),\ \ \mbox{where}\ \ \widehat{F}+\widehat{C}=1,\ \widehat{F}\ge 0,\ \widehat{C}\ge0,\\
&(0, C^*, H^*),\ \ \mbox{where}\ \  C^*=(1+s)/(1+sg),\ H^*=(1-g)/(1+sg).
\end{aligned}
\right.
\] 
The first two steady states always exist and unstable. The third one is a line of neutral equilibria, which always exists and stable if $g\ge 1$. The fourth one  exists and is stable if and only if $g<1$.  As noted in \cite{MX}, we call the case $g\ge 1$ by high conversion rate case and the case $g<1$ by low conversion rate case throughout this paper. 

\subsection{Previous works}

In the previous paper \cite{shigesada model},  Aoki {\it et al.} observed four different types of spreading behaviors depending on the parameter values, namely those of $a$, $s$ and $g$ in (\ref{fch-equation}) under the assumption that $d_c=d_h=1$. Furthermore, by mainly on numerical simulations, but also some formal analysis of the minimal speed of the traveling wave, they found the speed of the spreading fronts is determined by $\max\{2\sqrt{a},2\sqrt{1+s}\}$. In our previous work \cite{MX}, for the first time, we made mathematically rigorous studies to justify their claims with the same assumption on diffusion speeds. The following figures illustrate the shape of what they call the transient fronts.
\begin{center}
\begin{tikzpicture}[scale = 1.6]
\draw[<->](0,1.65)-- (0,0) -- (4.7,0) node[right] {$x$};
\draw[dotted, thick] (0,1) to [out=0,in=180] (1,0.8)--(3,0.8) to [out=0, in=180] (3.5,0)--(4.5,0.025);
\draw[very thick] (0,0.02) to [out=0,in=175] (1,0.3)--(3.1,0.3) to [out=5,in=175] (3.4,0.025)--(4.5,0.025);
\draw[thin] (0,0.02)--(3,0.02)to[out=0,in=180](3.5,1)--(4.5,1);
\draw[dashed] (3.25,1.5)--(3.25,-0.2);
\node[left] at (0,1) {\footnotesize{1}};
\node[left] at (0,0) {\footnotesize{0}};
\node[above] at (0.25,1) {\small{$F$}};
\node[above] at (2,0.4) {\small{$C$}};
\node[above] at (4,1) {\small{$H$}};
\node[right] at (3.25,0.5) {\textbf{$\rightarrow$}\ $c^*$};
\node[below] at (1.75,0) {\footnotesize{Final zone (for $F+C$)}};
\node[below] at (4.25,0) {\footnotesize{Leading edge}};
\node[below] at (2.5,-0.4) {\textbf{Figure 1:} High conversion rate case $g\ge 1$, $a> 1+s$.};
\end{tikzpicture}
\vspace{10pt}

\begin{tikzpicture}[scale = 1.6]
\draw[<->](0,1.65)-- (0,0) -- (4.7,0) node[right] {$x$};
\draw[dotted, thick] (0,1) to [out=0,in=180] (1,0.03)--(4.5,0.03);
\draw[very thick] (0,0.02) to [out=0,in=180] (1,1)--(3,1) to [out=0,in=180] (3.5,0.025)--(4.5,0.025);
\draw[thin] (0,0.02)--(3,0.02)to[out=0,in=180](3.5,1)--(4.5,1);
\draw[dashed] (3.25,1.5)--(3.25,-0.2);
\node[left] at (0,1) {\footnotesize{1}};
\node[left] at (0,0) {\footnotesize{0}};
\node[above] at (0.25,1) {\small{$F$}};
\node[above] at (2,1) {\small{$C$}};
\node[above] at (4,1) {\small{$H$}};
\node[right] at (3.25,0.5) {\textbf{$\rightarrow$}\ $c^*$};
\node[below] at (1.75,0) {\footnotesize{Final zone (for $F+C$)}};
\node[below] at (4.25,0) {\footnotesize{Leading edge}};
\node[below] at (2.5,-0.4) {\textbf{Figure 2:} High conversion rate case $g\ge 1$, $a<1+s$.};
\end{tikzpicture}
\vspace{10pt}

\begin{tikzpicture}[scale = 1.6]
\draw[<->](0,1.65)-- (0,0) -- (4.7,0) node[right] {$x$};
\draw[dotted, thick] (0,0.03) -- (2.8,0.03) to[out=0,in=210](3.25,0.5) to[out=350,in=170](3.6,0.03)--(4.5,0.03);
\draw[very thick] (0,1.3) --(2.7,1.3)to [out=0,in=180] (3.5,0.025)--(4.5,0.025);
\draw[thin] (0,0.8)--(3,0.8)to[out=0,in=180](3.5,1)--(4.5,1);
\draw[dashed] (3.25,1.5)--(3.25,-0.2);
\node[left] at (0,1) {\footnotesize{1}};
\node[left] at (0,0) {\footnotesize{0}};
\node[left] at (3.1,0.5) {\small{$F$}};
\node[above] at (1.5,1.3) {\small{$C$}};
\node[above] at (4,1) {\small{$H$}};
\node[right] at (3.25,0.6) {\textbf{$\rightarrow$}\ $c^*$};
\node[below] at (1.75,0) {\footnotesize{Final zone}};
\node[below] at (4.25,0) {\footnotesize{Leading edge}};
\node[below] at (2.5,-0.4) {\textbf{Figure 3:} Low conversion rate case $g<1$, $a> 1+s$.};
\end{tikzpicture}
\vspace{10pt}

\begin{tikzpicture}[scale = 1.6]
\draw[<->](0,1.65)-- (0,0) -- (4.7,0) node[right] {$x$};
\draw[dotted, thick] (0,0.02)--(4.5,0.02);
\draw[very thick] (0,1.3) --(2.7,1.3)to [out=0,in=180] (3.5,0.025)--(4.5,0.025);
\draw[thin] (0,0.8)--(3,0.8)to[out=0,in=180](3.5,1)--(4.5,1);
\draw[dashed] (3.25,1.5)--(3.25,-0.2);
\node[left] at (0,1) {\footnotesize{1}};
\node[left] at (0,0) {\footnotesize{0}};
\node[left] at (1.6,0.3) {\small{$F$}};
\node[above] at (1.5,1.3) {\small{$C$}};
\node[above] at (4,1) {\small{$H$}};
\node[right] at (3.25,0.6) {\textbf{$\rightarrow$}\ $c^*$};
\node[below] at (1.75,0) {\footnotesize{Final zone}};
\node[below] at (4.25,0) {\footnotesize{Leading edge}};
\node[below] at (2.5,-0.4) {\textbf{Figure 4:} Low conversion rate case $g<1$, $a<1+s$.};
\end{tikzpicture}
\end{center}

In order to make our motivation clearer, we give a brief explanation of the numerical observations by Aoki {\it et al.} \cite{shigesada model} and our previous theoretical results in \cite{MX} before stating our main results.
\begin{itemize}
\item[(1)] Theorem 1.5 in \cite{MX} indicates that the behaviors of solutions on the leading edge where hunter-gatherers have little contact with the farmers are almost the same between the high conversion rate case and the low conversion rate case (see Figure 1-4, Leading edge).

\item[(2)] Theorem 1.6 in \cite{MX} indicates that the spreading speed of the solution of \eqref{fch-equation} is always determined by $\max\{2\sqrt{a},2\sqrt{1+s}\}$.

\item[(3)] As showed in Figure 2 and Figure 4, in the case $a<1+s$, a wave of advance of initial farmers $F$ is not generated (see Theorem 1.9 and 1.10 in \cite{MX}). If, in addition that $g<1$ very small,  initial farmers $F$ disappear entirely, and converted farmers $C$ and  hunter-gatherers $H$ converge to a positive steady state (see Figure 4). Whereas, if $g\ge 1$, behind the wavefront, farmers have almost reached carrying capacity and hunter-gatherers have just  disappeared (see Figure 2,  Final zone).

\item[(4)] In the case $a>1+s$, an advancing wavefront of initial farmers $F$ is also generated (see Figure 1 and Figure 3). However, in \cite{MX}, we can only obtain an advancing wavefront formed by both initial farmers and converted farmers $G=F+C$. If, in addition that $g<1$, the waveform is a small peak with leading edge and trailing edge that converge to $0$ (see Theorem 1.8 in \cite{MX}).
\end{itemize}

In particular, for the case $a=1$ and $d_c=1$, by adding the both sides of $F$-equation and $C$-equation of \eqref{fch-equation}, we get
\begin{equation}\label{a=1 equation}
\left\{
\begin{aligned}
&\partial_tG= \partial_{xx}G+G(1+sH-G),\\
&\partial_tH=d_h \partial_{xx}H+bH(1-H-gG),
\end{aligned}
\right.\quad\quad\mbox{in}\ \ \mathbb{R}.
\end{equation}
The spreading speed of \eqref{a=1 equation} has been studied by Chen {\it et al.} in \cite{CT}.  Without constructing super and sub-solution and applying the comparison argument, they introduced
some Harnack type inequalities that can locate the spreading front of some general reaction-diffusion systems.

\begin{remark}\label{pic of F}
We would like to explain more on the numerical simulation showed in the case where $g\ge 1$ and $a<1+s$ (see Figure 2). The  numerical calculation was done in bounded domain and it would stop at the moment $T$ when the wavefronts of the solution reached a certain point $x_0\in[0,l]$. Therefore, one may find that  initial farmers $F$ still exist in the final zone. As a matter of fact,, we will prove in the present paper that the population density of initial farmers $F$ converges to $0$ uniformly for all $x\in\mathbb{R}$ (see Theorem \ref{th:profile c_c>c_f}). The sharp estimate given in Theorem \ref{th: bump} indicates that $F$ decreases to $0$ with polynomial order for $|x|\le t^{\frac{1}{2}}$, but with exponential order for $|x|>t^{\frac{1}{2}+\delta}$, $\delta>0$ very small. This explain the reason why we can observe a bump on profiles of $F$ and $C$.  
\end{remark}

\begin{remark}\label{rm: main result}
The goal of the present paper is to give rigorous justification to all of the above observations for the case $d_c\neq 1$ and $d_h\neq 1$.
The behaviors of solutions on the leading edge showed on Figure 1-4 can be justified by Theorem \ref{th:upper bound propagation}, which also provides an upper estimate on the spreading speed. On the other hand, the lower estimate on the spreading speed will be given in Theorem \ref{th:lower bound propagation}.
The behaviors of solutions
in the final zone showed on Figure 2 and 4 can be justified by Theorem \ref{th:profile c_c>c_f}  by dealing with the case $g\ge 1$ and $g<1$, respectively. 
\end{remark}

\subsection{Main results}
In this subsection, we will recall some previous works about the spreading properties of solutions of scalar KPP equation and present our main results which are concerned with the spreading speeds and the asymptotic profiles of solutions. Moreover, we will conclude how do our results justify the four different spreading behavior. 

Front propagation for scalar reaction-diffusion equations has been studied extensively and there exist vast literature on this theme. Early in 1937, Fisher \cite{F} and Kolmogorov, Petrovsky and Piskunov \cite{KPP} introduced a scalar reaction-diffusion equation with monostable nonlinearity as a model equation in population genetics, which studies the propagation of dominant gene in a homogeneous environment. In particular, \cite{KPP} made an important early analysis of the structure of the set of traveling waves for a special class of monostable reaction-diffusion equation, the so-called KPP equation. Application of reaction-diffusion equations to ecology was pioneered by Skellam \cite{S} in 1951. As regards the propagation of fronts for solutions with compactly supported initial data (that is, the so-called ``spreading front'') is concerned, the first rigorous mathematical results in multi-dimensional homogeneous environment were provided by Aronson and Weinberger \cite{AW}, for the case of monostable nonlinearity and also for bistable nonlinearity.

As regards the propagation of fronts for solutions of the reaction-diffusion system, there exists a rich literature on the cooperate system and the competition system, for which the comparison principle still hold.
However, apart from some recent works such as \cite{pp}, much less is known about the spreading properties for systems for which the comparison principle does not hold.

We first recall the classical spreading
result for the monostable scalar equation from Aronson and Weinberger \cite{AW}:
\begin{proposition}[\textbf{Spreading for the scalar KPP equation} (\cite{AW})]\label{prop of kpp equation}
Consider the so-called KPP equation

\begin{equation}\label{kpp equation}
\partial_tu(t,x)=d\partial_{xx}u(t,x)+f(u(t,x)),\ t>0,\ x\in\mathbb{R},
\end{equation}
wherein $f\in C^1(\mathbb{R})$ satisfies
$$f(0)=f(1)=0,\ f(u)>0\ \ \mbox{and}\ \ f'(u)\le f'(0)\ \ \mbox{for all}\ \ u\in (0,1).$$
Define $c^*(d,f)=2\sqrt{df'(0)}$. Then, for any nontrivial compactly supported and continuous initial data $u_0(x)$, the solution $u\equiv u(t,x;u_0(x))$ of (\ref{kpp equation}):
$$\underset{t\to+\infty}{\lim}\underset{|x| \le ct}{\sup}|1-u(t,x)|=0,\ 0<c<c^*(d,f),$$
and
$$\underset{t\to+\infty}{\lim}\underset{|x|\ge ct}{\sup}u(t,x)=0,\ c>c^*(d,f).$$
Moreover, the so-called spreading speed $c^*(d,f)$  coincides with the minimal speed of the traveling wave solution of admitted by (\ref{kpp equation}).
\end{proposition}


Next, we present our main results on the spreading properties of solutions of the system (\ref{fch-equation}). 
Throughout this paper, we define $c_f:=2\sqrt{a}$, $c_c:=2\sqrt{d_c(1+s)}$, and $c^*:=\max\{c_f,c_c\}$. 
The following Theorem \ref{th:upper bound propagation} and Theorem \ref{th:lower bound propagation}  can be regarded as an analogue of the well-known "hair-trigger effect" for scalar monostable equation \cite{AW}. Moreover, since $H_0(x)\equiv 1$,  one may find that the spreading speed of \eqref{fch-equation} is always determined by the larger value of $c_f$ and $c_c$, which is only determined by the parameters in the $F$-equation and $C$-equation.

Our first result is concerned with the analysis of the leading edge, which provides an upper estimate on the spreading speed. As we observed from Figure $1$-$4$, for all four cases, the behaviors of solutions are almost same on the leading edge. 

\begin{theorem}\label{th:upper bound propagation}
Let $(F,C,H)$ be the solution of \eqref{fch-equation} with the initial data satisfying \eqref{initial data}. Then for any $2\sqrt{a}<c_1$ and $c^*<c_2$, it holds:
\begin{equation}\label{upper F}
\limsup_{t\rightarrow\infty}\sup_{c_1 t\le|x|}F(t,x)=0,
\end{equation}
\begin{equation}\label{upper C H}
\limsup_{t\rightarrow\infty}\sup_{c_2 t\le |x|}\Big(C(t,x)+|1-H(t,x)|\Big)=0.
\end{equation}
\end{theorem}

The most difficult part of the analysis is the behaviors of the solutions in the final zone, where  hunter-gatherers, initial farmers and converted farmers heavily interact with each other. A large part of this paper is devoted to the analysis in this region.
Our second result deals with the propagation of farmers in the final zone, which provides a lower estimate on the spreading speed.

\begin{theorem}\label{th:lower bound propagation}
Let $(F,C,H)$ be the solution of \eqref{fch-equation} with the initial data  satisfying \eqref{initial data}. Then for any $c_1<c^*$, it holds:
\begin{equation}\label{lower F+C}
\liminf_{t\rightarrow\infty}\inf_{|x|\le c_1 t}\Big(F(t,x)+C(t,x)\Big)\ge1,
\end{equation}
\begin{equation}\label{lower H}
\limsup_{t\rightarrow\infty}\sup_{|x|\le c_1 t}H(t,x)\le \max\{0,1-g\}.
\end{equation}
\end{theorem}

\begin{remark}\label{rm: c_f and c_c}
The values $2\sqrt{a}$ and $2\sqrt{d_c(1+s)}$ indicate the spreading speed of initial farmers and converted farmers respectively. More precisely, if we let $C\equiv 0$ in \eqref{fch-equation}, then the spreading speed of initial farmers directly comes from the single equation 
$$\partial_tF= \partial_{xx}F+aF(1-F-C).$$ 
On the other hand, if we let $F\equiv 0$ in \eqref{fch-equation}, then we get 
\begin{equation*}
\left\{
\begin{aligned}
&\partial_tC=d_c \partial_{xx}C+C(1-F-C)+sH(F+C),\\
&\partial_tH=d_h \partial_{xx}H+bH(1-H-g(F+C)),
\end{aligned}
\right.\quad\quad\mbox{in}\ \ \mathbb{R}.
\end{equation*}
The main results in \cite{CT} indicates the spreading speed of converted farmers is equal to $2\sqrt{d_c(1+s)}$.
\end{remark}

At last, we show our main results about the asymptotic profiles of solutions in the final zone. The precise results will be presented in two cases: the case $a>d_c(1+s)$ and the case $a<d_c(1+s)$, respectively. The critical case $a=d_c(1+s)$ will not be discussed in the present paper. Moreover, for the each case, the profiles of solutions are also depending on the value of $g$. For the high conversion rate case, we not only show that all of the original hunter-gatherers convert to farmers at last,  but also investigate the explicit profiles ({\it the bump phenomena}) of the $F$-component and $C$-component in the final zone. For the low conversion rate case, we find that the $F$ converges to $0$ global uniformly.  Then,
we can investigate how the $C$ and $H$ behave in the finial zone by considering the dynamics of the underlying ODE system:
\begin{equation}\label{Ode system}
\left\{
\begin{aligned}
&C_t=C(1-C)+sCH,\\
&H_t=bH(1-H-gC).
\end{aligned}
\right.
\end{equation}
We expect the solution of the PDE system (\ref{fch-equation}) to uniformly converge to the equilibrium $(C^*,H^*)$ in the final zone as $t\to+\infty$. We prove the conjecture by the established fact that a strict Lyapunov function exists.

\begin{theorem}\label{th:profile c_c>c_f}
Assume $d_c(1+s)>a$. Let $(F,C,H)$ be the solution of \eqref{fch-equation} with the initial data satisfying \eqref{initial data}. For any $g>0$, one has
\begin{equation}\label{F unif to 0 for all x}
\limsup_{t\rightarrow\infty}\sup_{x\in{\mathbb R}}F(t,x)=0.
\end{equation}
Moreover, it holds:
\begin{itemize}
\item[(1)] {\bf{(high conversion rate case)}} if $g\ge 1$, then for any $c_1<c^*$, one has
\begin{equation}\label{asympt C H g>1 a<1+s}
\limsup_{t\rightarrow\infty}\sup_{|x|\le c_1 t}\Big(|C(t,x)-1|+H(t,x)\Big)=0.
\end{equation}
\item[(2)] {\bf{(low conversion rate case)}} if $g<1$, then for any $c_1<c^*$, one has
\begin{equation}\label{asympt C H g<1 a<1+s}
\limsup_{t\rightarrow\infty}\sup_{|x|\le c_1 t}\Big(|C(t,x)-C^*|+|H(t,x)-H^*|\Big)=0.
\end{equation}
\end{itemize}
\end{theorem}

\begin{remark}\label{rm: no condition on g and b}
Different with Theorem 1.10 in \cite{MX}, here we do not need any hypothesis on values of $g$ and $b$, due to a more involved analysis on the limit system.
\end{remark}

Our next result reveals the reason why bumps have been observed on the profiles of $F$ and $C$ when $d_c(1+s)>a$ and $g>1$. Roughly speaking, the profile of the solution in the region $\vert x \vert \leq \varepsilon_* t$ is more similar like the Heat equation type.
\begin{theorem}[Bump phenomenon]\label{th: bump}
Assume $d_c(1+s)>a$ and $g>1$. Let $(F,C,H)$ be the solution of \eqref{fch-equation} with the initial data satisfying \eqref{initial data}. Denote
$$
k^*:=\min\left(\frac{1}{2d_c},\frac{d_c}{2}\right), \quad d^*:=\max(1,d_c).
$$
Then for $\varepsilon_*>0$ small enough and $0<\theta<\frac{1}{2}$,  there exist $C_2>C_1>0$, $C_3>0$, and $T>0$ such that both
\begin{equation}\label{profile F the solution small x}
C_1t^{-\frac{1}{2}}e^{-\frac{x^2}{4t}}\le F(t,x)\le C_2t^{-k^*}e^{-\frac{x^2}{4d^*t}},\ 
\end{equation}
\begin{equation}\label{profile of C the solution small x}
 1-C_2t^{-k^*}e^{-\frac{x^2}{4d^*t}}\le C(t,x)\le  1-C_1t^{-\frac{1}{2}}e^{-\frac{x^2}{4t}}+C_3t^{-(1+\theta)},
\end{equation}
hold for $t\geq T$, $\vert x\vert \leq \varepsilon_*t$. 
\end{theorem}
It is obviously that $F(t,x)\to 0$ and $C(t,x)\to 1$ with polynomial orders for $|x|\le t^{\frac{1}{2}}$, but polynomial orders for $|x|\ge t^{\frac{1}{2}+\delta}$ with any $\delta>0$.

The following theorem, which provides an estimate of $F$ on the wavefront, indicates the exsitence of a small peak of $F$ on the wavefront.
\begin{theorem}\label{th:front peak F}
Let $(F,C,H)$ be the solution of \eqref{fch-equation} with initial data $(F_0,C_0,H_0)$ satisfying \eqref{initial data}. If $a>1+s$, then there exists $\varepsilon_0>0$ such that
\begin{equation}\label{front peak F 2}
\limsup_{t\rightarrow\infty}\sup_{c_1 t\le |x|}F(t,x)\ge \varepsilon_0\quad{\rm for\ all}\quad c_1<c^*.
\end{equation}
\end{theorem}

\section{Upper estimates on the spreading speeds}
In this section, we deal with the spreading properties of solutions of (\ref{fch-equation}) on the leading edge and complete the proof of Theorem \ref{th:upper bound propagation}. Since there is little interaction between farmers and hunter-gatherers on the leading edge, the analysis of this zone is rather straightforward. 

We first deal with the analysis of $F$ and $C$ on the leading edge and begin the proof with some simple upper estimates on the spreading speed. From $C\ge 0$ for all $(t,x)\in\mathbb{R}^+\times\mathbb{R}$, by a well-known super-solution 
\begin{equation}\label{well known sup slo}
\bar{F}:=\min\{A_1e^{-\sqrt{a}(|x|-2\sqrt{a}t)},1\},
\end{equation}
one has 
\begin{equation}\label{eq of upper estimate of F}
\underset{t\to+\infty}{\lim}\underset{|x|\ge ct}{\sup}F(t,x)\le \underset{t\to+\infty}{\lim}\underset{|x|\ge ct}{\sup}\bar{F}(t,x)=0\ \ \mbox{for all}\ \ c>c^*\ge 2\sqrt{a},
\end{equation}
provided that $A_1$ is large enough such that $\bar{F}(0,x)\ge F(0,x)$. This proves \eqref{upper F} in Theorem \ref{th:upper bound propagation}.

To get the estimate of  $C$, we need to construct another suitable super-solution. The following discussion will be divided into two cases $a\ge d_c(1+s)$ and $a<d_c(1+s)$.

\noindent{\bf{Case 1: $a\ge d_c(1+s)$}}

Note that $H(t,x)\le 1$ for all $(t,x)\in\mathbb{R}^+\times\mathbb{R}$. We define 
$$\bar C_1:=\min\{A_2e^{-{\sqrt\frac{(1+s)}{d_c}}(x-(2\sqrt{a}+\varepsilon)t)},1+2s\}$$
with sufficiently small $\varepsilon>0$ and $A_2\ge 1+2s$, and $\bar F_1=\bar F$ defined as \eqref{well known sup slo}. Then, we see $\bar F_1=o(\bar C_1)$ as $t\to\infty$ for $x\ge (2\sqrt{a}+\varepsilon)t$ if $d_c\ge 1$. Moreover, since $A_2\ge 1+2s$, it suffices to check $\bar C_1$ satisfies the super-solution inequality for $x\ge (2\sqrt{a}+\varepsilon)t$.
By some straightforward computations, there exist $T_1>0$ such that for $t\ge T_1$, one has
\begin{equation*}
\begin{aligned}
&\partial_t\bar C_1-d_c\partial_{xx}\bar C_1-\bar C_1(1-\bar C_1-F)-sH(\bar C_1+F)\\
\ge& \partial_t\bar C_1-d_c\partial_{xx}\bar C_1-(1+s)\bar C_1-s\bar F_1\\
=& \Big(\frac{2\sqrt{a(1+s)}}{\sqrt{d_c}}+\varepsilon\sqrt{\frac{1+s}{d_c}}-2(1+s)\Big)\bar C_1-s\bar F_1\ge0, 
\end{aligned}
\end{equation*}
provided that $d_c\ge 1$ and $T_1$ is large enough.

On the other hand, if $d_c<1$, we consider 
$$\hat C_1:=\min\{A_2e^{-\sqrt{d_c(1+s)}(x-(2\sqrt{a}+\varepsilon)t)},1+2s\}.$$
It still holds that $\bar F_1=o(\hat C_1)$ as $t\to\infty$ for $x\ge (2\sqrt{a}+\varepsilon)t$.
Then, by some straightforward computations, there exist $T_2>0$ such that, for $x\ge (2\sqrt{a}+\varepsilon)t$ and $t\ge T_2$, one has
\begin{equation*}
\begin{aligned}
&\partial_t\hat C_1-d_c\partial_{xx}\hat C_1-\hat C_1(1-\hat C_1-F)-sH(\hat C_1+F)\\
\ge& \partial_t\hat C_1-d_c\partial_{xx}\hat C_1-(1+s)\hat C_1-s\bar F_1\\
=& (2\sqrt{ad_c(1+s)}+\varepsilon\sqrt{d_c(1+s)}-(d^2_c+1)(1+s)\Big)\hat C_1-s\bar F_1\ge0, 
\end{aligned}
\end{equation*}
provided that $T_2$ is large enough.

\noindent{\bf{Case 2:$a<d_c(1+s)$}}

First, one can check that $\bar F_2:=\min\{A_3e^{-\sqrt{d_c(1+s)}(x-2\sqrt{d_c(1+s)}t)},1\}$ satisfies
$$\partial_t\bar F_2-\partial_{xx}\bar F_2-a(1-\bar F_2)\ge 0,$$
and hence $F(t,x)\le \bar F_2(t,x)$ for $(t,x)\in\mathbb{R}^+\times\mathbb{R}$, provided that $A_3$ is sufficiently large.

For the case $d_c\ge 1$, we define 
$$\bar C_2:=\min\{A_4e^{-{\sqrt\frac{(1+s)}{d_c}}(x-(2\sqrt{d_c(1+s)}+\varepsilon)t)},1+2s\}$$
with sufficiently small $\varepsilon>0$ and $A_4\ge 1+2s$.
Then, we see $\bar F_2=o(\bar C_2)$ as $t\to\infty$ for $x\ge (2\sqrt{d_c(1+s)}+\varepsilon)t$ if $d_c\ge 1$. 

We still only check $\bar C_2$ satisfies the super-solution inequality for $x\ge (2\sqrt{d_c(1+s)}+\varepsilon)t$.
By some straightforward computations, there exist $T_3>0$ such that, for $x\ge (2\sqrt{d_c(1+s)}+\varepsilon)t$ and $t\ge T_3$, one has
\begin{equation*}
\begin{aligned}
&\partial_t\bar C_2-d_c\partial_{xx}\bar C_2-\bar C_2(1-\bar C_2-F)-sH(\bar C_2+F)\\
\ge& \partial_t\bar C_2-d_c\partial_{xx}\bar C_2-(1+s)\bar C_2-s\bar F_2\\
=& \varepsilon\sqrt{\frac{1+s}{d_c}}\bar C_2-s\bar F_2\ge0, 
\end{aligned}
\end{equation*}
provided that $T_3$ is large enough.

On the other hand, if $d_c<1$, we consider 
$$\hat C_2:=\min\{A_4e^{-\sqrt{d_c(1+s)}(x-(2\sqrt{d_c(1+s)}+\varepsilon)t)},1+2s\}.$$
It still holds that $\bar F_2=o(\hat C_2)$ as $t\to\infty$ for $x\ge (2\sqrt{d_c(1+s)}+\varepsilon)t$.
Then, there exist $T_4>0$ such that, for $x\ge (2\sqrt{d_c(1+s)}+\varepsilon)t$ and $t\ge T_4$, one has
\begin{equation*}
\partial_t\hat C_2-d_c\partial_{xx}\hat C_2-\hat C_2(1-\hat C_2-F)-sH(\hat C_2+F)\ge\varepsilon\sqrt{d_c(1+s)}\hat C_2-s\bar F_2\ge0, 
\end{equation*}
provided that $T_4$ is large enough.

The upper estimates on $C$ follow immediately from the super-solution constructed above.
By setting $A_2$({\it{resp.}}$A_4$) large such that 
$$C(T_{1},x)\le\bar C_1(T_1,x)({\it{resp.}}C(T_{2},x)\le\hat C_1(T_2,x), C(T_{3},x)\le\bar C_2(T_3,x), C(T_{4},x)\le\hat C_2(T_4,x)),$$
and applying the comparison principle, one has 
\begin{equation}\label{eq of upper estimate of C}
\underset{t\to+\infty}{\lim}\underset{|x|\ge ct}{\sup}C(t,x)=0\ \ \mbox{for all}\ \ c>c^*,
\end{equation}
since $\varepsilon>0$ can be chosen arbitrarily small.
This proves the first part of statement \eqref{upper C H} in Theorem \ref{th:upper bound propagation}.

At the end of this section, we deal with the upper estimate of the $H$ on the leading edge. Let $\varepsilon>0$ be fixed from the super-solution of $C$(see the definition of $\bar C_1$, $\hat C_1$, $\bar C_2$, and $\hat C_2$), and define
\begin{equation}\label{sub solution of H on leading edge}
\underline H:=\max\{1-A_5e^{-\lambda_h(x-(c^*+2\varepsilon)t)},0\}\ \ \text{with}\ \ \lambda_h\le \min\{\frac{c^*}{d_h},\sqrt{\frac{1+s}{d_c}},\sqrt{d_c(1+s)}\}.
\end{equation}
We choose $A_5$ large enough such that $\underline H(t,x)=0$ for $x\le (c^*+2\varepsilon)t$ and $t\ge 0$. Moreover, by the condition of $\lambda_h$ in \eqref{sub solution of H on leading edge} and super-solutions of $F$ and $C$ constructed above, there exist $T_5>0$ such that $1-\underline H\ge g(F+C)$ for $x\ge (c^*+2\varepsilon)t$ and $t\ge T_5$.

By some straightforward computations, for $x\ge (c^*+2\varepsilon)t$, one has 
\begin{equation*}
\begin{aligned}
&\partial_t\underline H-d_h\partial_{xx}\underline H-b(1-\underline H-g(F+C))\\
\le& (d_h\lambda_h^2-(c^*+2\varepsilon)\lambda_h)(1-\underline H)-b\underline H(1-\underline H-g(C+F))\le 0.
\end{aligned}
\end{equation*}
Hence, since $\varepsilon>0$ can be chosen arbitrarily small, we obtain that 
$$\underset{t\to\infty}{\lim}\underset{|x|\ge ct}{\inf}H(t,x)\ge\underset{t\to\infty}{\lim}\underset{|x|\ge ct}{\inf}\underline H(t,x)=1\ \ \mbox{for all}\ \ c>c^*.$$
Hence, we find that $H$is always able to stay positive outside of the farmers' range, which completes the proof of Theorem \ref{th:upper bound propagation}.

\section{Lower estimates on the spreading speed for the case $d_c(1+s)>a$}
In this section, we deal with the lower estimate on the spreading speeds of solutions of system (\ref{fch-equation}) for the case  $d_c(1+s)>a$. 

\begin{proposition}\label{prop 5 of log delay}
Assume  $d_c(1+s)>a$. Then for any $R>0$, it holds:
\begin{equation}\label{location of C a<1+s}
\underset{t\to+\infty}{\liminf}\inf_{|x|\le R}\ C\Big(t,x+c^*t-\frac{3d_c}{c^*}\ln t\Big)>0,\quad\text{and}\quad\underset{t\to+\infty}{\limsup}\sup_{|x|\le R}\ C\Big(t,x+c^*t-\frac{3d_c}{c^*}\ln t\Big)<1.
\end{equation}
\begin{equation}\label{location of H a<1+s}
\underset{t\to+\infty}{\liminf}\inf_{|x|\le R}\ H\Big(t,x+c^*t-\frac{3d_c}{c^*}\ln t\Big)>0,\quad\text{and}\quad\underset{t\to+\infty}{\limsup}\sup_{|x|\le R}\ H\Big(t,x+c^*t-\frac{3d_c}{c^*}\ln t\Big)<1.
\end{equation}
\end{proposition}

\begin{proposition}\label{prop: lower estimate on C a<1+s}
Assume $d_c(1+s)>a$. Then for any $c_f<c_1<c_2<c_c$, it holds
\begin{equation}\label{C>=1 a<1+s}
\liminf_{t\to\infty}\inf_{c_1t\le |x|\le c_2t}C(t,x)\ge 1,\quad\text{and}\quad \limsup_{t\to\infty}\sup_{c_1t\le |x|\le c_2t}H(t,x)\le \max\{0,1-g\}.
\end{equation}
Moreover, there exists $\gamma_c>0$ such that
\begin{equation}\label{C>1-e^-t}
\inf_{c_1t\le|x|\le c_2t}C(t,x)\ge 1-e^{-\gamma_ct}\quad\text{for all}\quad t>T. 
\end{equation}
\end{proposition}
\begin{proof}{\it of Proposition \ref{prop: lower estimate on C a<1+s}.}
It suffices to show  there exists $\varepsilon>0$ such that
\begin{equation}\label{aa}
\liminf_{t\to+\infty}\inf_{c_1t\le|x|\le c_2t}C(t,x)\ge\varepsilon.
\end{equation}
Indeed, we can claim that
\begin{claim}\label{cl:aa}
The statement \eqref{aa} implies that \eqref{C>=1 a<1+s} holds true.
\end{claim}
\begin{proof}{\it of Claim \ref{cl:aa}.}
Let us arbitrarily choose sequences  $\{c_n\}_{n\ge 0}\subset[c'_1,c'_2]$ with $c_1<c_1'<c_2'<c_2$, $\{t_n\}_{n\ge 0}\subset\mathbb{R}_+$ with $t_n\to+\infty$ as $n\to+\infty$, and $\{x_n\}_{n\ge 0}\subset\mathbb{R}$ with $|x_n|\le c_nt_n$. Next, we prove 
$$\lim_{n\to+\infty}H(t_n,x_n)\le \max\{0,1-g\}\ \ \mbox{and}\ \ \lim_{n\to+\infty}C(t_n,x_n)\ge1.$$
We consider the limit functions
$$\lim_{n\to+\infty}F(t_n+t,x_n+x)=F_{\infty}(t,x),$$
$$\lim_{n\to+\infty}C(t_n+t,x_n+x)=C_{\infty}(t,x),$$
$$\lim_{n\to+\infty}H(t_n+t,x_n+x)=H_{\infty}(t,x),$$
which converge locally uniformly to $(F_{\infty}, C_{\infty}, H_{\infty})$.
The super-solution \eqref{well known sup slo} of $F$  yields that $F_{\infty}(0,0)=0$, which also implies $F_{\infty}(t,x)\equiv 0$ for all $(t,x)\in\mathbb{R}\times\mathbb{R}$ by the strong maximum principle. Therefore, $(C_{\infty},H_{\infty})$ is
an entire solution of the system 
\begin{equation}\label{limit system C H}
\left\{
\begin{aligned}
&\partial_tC_{\infty}=d_c \partial_{xx}C_{\infty}+C_{\infty}(1-C_{\infty})+sC_{\infty}H_{\infty},\\
&\partial_tH_{\infty}=d_h\partial_{xx} H_{\infty}+bH_{\infty}(1-gC_{\infty}-H_{\infty}).
\end{aligned}
\right.
\end{equation}

 Note that \eqref{aa} implies  $C_{\infty}(t,x)\ge\varepsilon$ for all $(t,x)\in\mathbb{R}\times\mathbb{R}$. Moreover, one may find that $1-(1-\varepsilon)e^{-\varepsilon(t+t_0)}$ is a sub-solution of $C_{\infty}(t,x)$ for all $x\in\mathbb{R}$ and $t>-t_0$ where $t_0\in\mathbb{R}_+$. By passing $t_0\to+\infty$, one obtains that $C_{\infty}(0,x)\ge 1$. Since the sequences $\{t_n\}_{n\ge 0}$, $\{c_n\}_{n\ge 0}$ and $\{x_n\}_{n\ge 0}$ are chosen arbitrarily, one can conclude that 
\begin{equation*}
\liminf_{ t \to \infty}\inf_{c_1't\le|x|\le c_2't}C(t,x)\ge 1.
\end{equation*}
Next, we show that $H_{\infty}(t,x)\le \max\{0,1-g\}$.  Indeed, since $C_{\infty}(t,x)\ge 1$ for all $(t,x)\in\mathbb{R}\times\mathbb{R}$, one may find $H_{\infty}$ is a sub-solution of $\widetilde{H}_{t_0}$ which satisfies
\begin{equation*}
\left\{
\begin{aligned}
&\partial_t\widetilde{H}_{t_0}=d_h\partial_{xx}\widetilde{H}_{t_0}+b\widetilde{H}_{t_0}(1-g-\widetilde{H}_{t_0}),\quad x\in{\mathbb R},\ t>-t_0,\\
&\widetilde{H}_{t_0}(-t_0,x)=H_{\infty}(-t_0,x),\quad x\in{\mathbb R}.
\end{aligned}
\right.
\end{equation*}
It is clear that 
\[H_{\infty}(0,x)\le\underset{t_0\rightarrow\infty}{\lim}\,\widetilde{H}_{t_0}(0,x)=\max\{0, 1-g\}.\] 
Thus it holds that 
\[
\limsup_{t\rightarrow\infty}\sup_{c_1't\le|x|\le c_2't}H(t,x)\le \max\{0,1-g\}.
\]
The proof of the claim is complete since $c_f<c_1<c_1'<c_2'<c_2<c_c$ are chosen arbitrarily. 
\end{proof}
\vspace{10pt}

Now, let us go back to the proof of \eqref{aa}. To do this, we assume by contradiction that there exist sequences $\{t_n\}_{n\ge 0}\subset\mathbb{R}_+$ with $t_n\to+\infty$ and $\{x_n\}_{n\ge 0}\subset\mathbb{R}$ with $c_1't_n\le|x_n|\le c'_2t_n$, such that $C(t_n,x_n)\le 1/n$.
Let us introduce the principal eigenfunction $\phi_R(x)>0$ as
\begin{equation}\label{eigenfunction}
\left\{
\begin{aligned}
&-\partial_{xx}\phi_R=\mu_R\phi_R\ \ \mbox{in}\ \ B_R,\\
&\phi_R=0\ \ \mbox{on}\ \ \partial B_R,
\end{aligned}
\right.
\end{equation}
that is
normalized so that $\underset{x\in B_R}{\sup}\,\phi_R(x)=1$, and extend it by $0$ outside of the ball $B_R$. The eigenvalue $\mu_R$ is positive and satisfying $\mu_R=\mu_1R^{-2}$. 

Then, we consider a  sub-solution of the $C$-equation as 
$\varphi_R(t,x):=\eta\phi_R(x)(1+e^{-\tau_Rt})$ with $\tau_R<\sqrt{a}(c_1-c_f)$.
Note that, by the super-solution \eqref{well known sup slo} of $F$, one has $F(t,x)\le A_1e^{-\sqrt{a}(c_1'-c_f)t}$ for $|x|\ge c_1't$. 
Then one can check that, there exists $T>0$ such that 
$$\eta\partial_t \varphi_R- \eta d_c\partial_{xx}\varphi_R-\eta\varphi_R(1-\eta\varphi_R-F)\le 0,\quad\text{for all}\quad t\ge T,$$
provided that $\eta$ is small enough and $R$ is large enough.
In Proposition \ref{prop 5 of log delay}, we proved that, for all $R'>R>0$, there exists $\varepsilon'>0$ such that
$$\liminf_{t\to+\infty}\inf_{x\in B_{R'}}\ C\Big(t,x+c_ct-\frac{3d_c}{c_c}\ln t\Big)>\varepsilon'.$$
Hence, for any $t'>T_0$, one can choose $\eta$ small enough such that
\begin{equation*}
C(t',x)\ge 2\eta\varphi_R\Big(x-(c_ct'-\frac{3d_c}{c_c}\ln t')\Big)\ \ \mbox{for all}\ \ x\in\mathbb{R}.
\end{equation*}
Then by applying the comparison principle, one obtains
$$C(t,x)\ge \eta(1+e^{-\tau_Rt})\varphi_R\Big(x-(c_ct'-\frac{3d_c}{c_c}\ln t')\Big)\ \ \mbox{for all}\ \ t>t', \ \ x\in\mathbb{R}.$$
This implies that, for any $t'\ge T_0$ and $t\ge t'$, it holds
\begin{equation}\label{eq 7}
C\Big(t,(c_ct'-\frac{3d_c}{c_c}\ln t')\Big)\ge \eta\phi_R(0).
\end{equation}
Moreover, since $c'_1<c'_2<c_c$ and $c'_1t_n\le|x_n|\le c'_2 t_n$, for each large enough $n$, one can find $t_n'\in[T_0,t_n)$ such that 
$$x_n=\Big(c_ct_n'-\frac{3d_c}{c_c}\ln t_n'\Big).$$
Thus, from the estimate (\ref{eq 7}), one gets
$G
C(t_n,x_n)\ge \eta\phi_R(0)$,
which contradicts that $C(t_n,x_n)\le 1/n\to 0$ as $n\to+\infty$. 
Therefore, the proof of \eqref{aa} is complete.

The super-solution \eqref{well known sup slo} showed that
$$\inf_{c_1t\le|x|\le c_2t}F(t,x)\le A_1e^{-\sqrt{a}(c_1-c_f)t}\quad\text{for all}\quad t>0.$$
Then, by using the same argument in the proof of Lemma 1.11 in \cite{CXZ} with $b=1$, we can obtain the estimate \eqref{C>1-e^-t}. 
\end{proof}


\section{Asymptotic profiles for the case \bf{$d_c(1+s)>a$}}
We have already shown that the propagation of farmers occurs with the speed $c^*$. However, whether the profiles of solutions converges to the steady states $(\widehat{F},\widehat{C},0)$ or $(0,C^*,H^*)$ are still unknown.  In this section, we mainly deal with the asymptotic profiles of solutions in the final zone for the case $d_c(1+s)>a$. According to the numerical work of Aoki {\it et al.}, the profiles of solutions in the final zone are suppose to be different between the cases $g\ge 1$ and $g<1$.  In particular, a bump has been observed in the case $g> 1$. 

The justification of the numerical results will be split into several parts. The first part contributes to an upper estimate on $F$ and lower estimate on $C$, which holds for all $g>0$. In particular, we will show that $F$ uniformly converges to $0$ for all $x\in\mathbb{R}$, which will play an important role in investigating the asymptotic profiles of $C$ and $H$ in the final zone. 

\subsection{The upper estimate on $F$ and the lower estimate on $C$}
Throughout this subsection, we let $\max\{c_f,2\sqrt{d_c}\}<c_0<c_c$ be fixed.  First of all, we will show $F$ uniformly converges to $0$ for all $x\in\mathbb{R}$ by comparing with the solution of a system as follows:
\begin{equation}\label{critical competition system}
\left\{
\begin{aligned}
&u_t=u_{xx}+au(1-u-v),\\
&v_t=d_cv_{xx}+v(1-v-u).\\
\end{aligned}
\right.
\end{equation}
The system \eqref{critical competition system} is the so-called critical Lotka-Volterra competition system, for which the Cauchy problem with two compactly supported initial data $(u_0(x),v_0(x))$ has been well studied in \cite{AX}. More precisely, since $u,v$ have a same competition strength, Alfaro {\it{et al.}} found that the fast species dominates the habitat at last, but the slow one becomes extinct. In other words, if $a>d_c$, then $v$ uniformly converges to $0$; if $a<d_c$, then $v$ uniformly converges to $0$. Note that, in our case, the initial data is not supposed to be two compactly supported functions. However, an a priori estimate on $C(t,x)$ from Proposition \ref{prop: lower estimate on C a<1+s} roughly implies that $C$ is the fast species.

Let us recall the competitive comparison principle. Define the operators
\begin{equation}\label{def of sup and sub sol}
N_1[u,v]:=\partial_t u-\partial_{xx} u-u(1-u-v)\quad \text{ and } \quad  N_2[u,v]:=\partial_tv-d_c\partial_{xx}v-v(1-v-u)
\end{equation}
on the domain $
\Omega:=(t_1,t_2)\times(x_1,x_2)$ with  $0\le t_1<t_2\le + \infty$ and $-\infty\le x_1<x_2\le+\infty$. A super-solution is a pair  
$$(\overline{u},\underline{v})\in \Big[C^1\Big((t_1,t_2),C^2((x_1,x_2))\Big)\cap C_b\left(\overline \Omega\right)\Big]^2$$
satisfying
$$
N_1[\overline{u},\underline{v}]\geq 0 \quad  \text{and}\quad  N_2[\overline{u},\underline{v}]\leq 0\;  \text{ in } \Omega.
$$
Similarly, a sub-solution $(\underline u, \overline v)$ satisfies  $N_1[\underline{u},\overline{v}]\leq 0$ and $N_2[\underline{u},\overline{v}]\geq 0$. 

\begin{proposition}[Comparison Principle]\label{prop: cp}
Let $(\overline{u},\underline{v})$ and $(\underline{u},\overline{v})$ be a super-solution and sub-solution satisfying \eqref{def of sup and sub sol}, respectively. If
$$
\overline{u}(t_1,x)\ge \underline{u}(t_1,x) \quad \text{and} \quad \underline{v}(t_1,x)\le\overline{v}(t_1,x),\quad\text{for all } x\in (x_1,x_2),
$$
and, for $i=1,2$, 
$$
\overline{u}(t,x_i)\ge \underline{u}(t,x_i) \quad \text{and} \quad \underline{v}(t,x_i)\le\overline{v}(t,x_i),\quad\text{for all } t\in(t_1,t_2),
$$
then, it holds
$$ 
\overline{u}(t,x) \ge\underline{u}(t,x) \quad \text{ and } \quad \underline{v}(t,x)\le\overline{v}(t,x),\quad  \text{for all } (t,x)\in\Omega.
$$
 If $x_1=-\infty$ or $x_2=+\infty$, the corresponding boundary conditions can be omitted.
\end{proposition}

\subsubsection{Construction of the super-solution}
We now construct an super-solution in the region $\vert x\vert \le c_0 t$. Fix $c_1>c_c$ and introduce $V_{1}$ as a traveling wave solution with speed $c_1$ solving
\begin{equation}\label{single-V1}
\left\{
\begin{aligned}
&d_cV_1''+c_1V_1'+V_1(1-V_1)=0,\\
&V_1(-\infty)=1,\  V_1(\infty)=0.
\end{aligned}
\right.
\end{equation} 
As well-known, $V_1'<0$ and there exist $\lambda_1>0$ and $M _1>0$ such that
\begin{equation}\label{estimate of V at infinity-un}
 1-V_1(\xi)\sim M _1 e^{\lambda _1 \xi}\; \text{ as } \xi \to -\infty.
\end{equation}
For $T>0$, we will work in the domain 
\begin{equation}\label{def:omega-un}
\Omega _0(T):=\{(t,x)\in(T,\infty)\times \mathbb{R}: \vert x\vert <c_0t\}.
\end{equation}
It turns out that the construction of the super-solution is highly dependent on the value of $d_c$.

\medskip

\noindent{$\bullet$  \bf The case $d_c\le 1$.} We introduce a solution of the Cauchy problem of heat equation
\begin{equation}
\label{def-s}
\left\{
\begin{aligned}
 \alpha_t&= \alpha_{xx},\\
\alpha(0,x)&=\alpha_0(x):=B_1e^{-q|x|},
\end{aligned}
\right.
\end{equation}
and consider a super-solution $(\overline u,\underline v)$ with positive parameters in form of 
\begin{equation}\label{definition of super sol system}
\left\{
\begin{aligned}
\overline u(t,x)&:=t^{\frac{1-d_c}{2}}(1-e^{-\tau t})\alpha(t,x),\\
\underline v(t,x)&:=V_{1}(x-c_1t)+V_{1}(-x-c_1t)-1-\overline{u}(t,x).
\end{aligned}
\right.
\end{equation}

\medskip

\noindent{$\bullet$ \bf The case $d_c\ge 1$.} We introduce a solution of the Cauchy problem of heat equation
\begin{equation}
\label{def-s*}
\left\{
\begin{aligned}
 \alpha_t &= d_c\alpha_{xx},\\
\alpha(0,x)&=\alpha_0(x):=B_1e^{-q|x|},
\end{aligned}
\right.
\end{equation}
and consider a super-solution $(\overline u,\underline v)$ with positive parameters in form of 
\begin{equation}\label{definition of super sol system d>1}
\left\{
\begin{aligned}
\overline{u}(t,x)&:=t^{\frac{d_c-1}{2d_c}}(1-e^{-\tau t})s(t,x),\\
\underline{v}(t,x)&:=V_{1}(x-c_1t)+V_{1}(-x-c_1t)-1-\overline{u}(t,x).
\end{aligned}
\right.
\end{equation}
Note that, \eqref{def-s}-\eqref{definition of super sol system} and \eqref{def-s*}-\eqref{definition of super sol system d>1} obviously coincide when $d_c=1$.

\begin{proposition}[Super-solutions]\label{prop:inequality super solution} The following holds.
\begin{itemize}
\item[$(i)$] Assume $d_c\le 1$. Let $0<q<\min\{\frac{c_0}{2},\frac{\lambda_1c_1}{c_0}\}$ and $0<\tau<\lambda_1(c_1-c_0)$ be given. Then there exists $T^{*}>0$ such that,  for all $B_1>0$, $(\overline u,\underline v)$ is  a super-solution in the domain $\Omega_0(T^{*})$.

\item[$(ii)$] Assume $d_c\ge 1$.  Let $0<q<\min\{\frac{c_0}{2d_c},\frac{\lambda_1c_1}{c_0}\}$ and $0<\tau<\lambda_1(c_1-c_0)$ be given. Then there exists $T^{*}>0$ such that,  for all $B_1>0$, $(\overline u,\underline v)$,  is  a super-solution in the domain $\Omega_0(T^{*})$.
\end{itemize}
\end{proposition}

\begin{proof}{\it of Proposition \ref{prop:inequality super solution}.} Since the super-solutions  are even functions, it is enough to check the super-solution inequality for $t\geq T$  and $0\leq x< c_0t$.  For ease of notations, we denote  $\xi_{\pm}:=\pm x-c_1t$. Since $\xi_-\leq -c_1t$ and $\xi_+\leq -(c_1-c_0)t$, it follows  from $V_1'<0$ and \eqref{estimate of V at infinity-un}  that there exist $C_->0$ and $C_+>0$ such that, for $T>0$ large enough,
\begin{equation}\label{estimate of 1-V super sol}
1-V_1(\xi_-)\leq C_-e^{-\lambda_1c_1t}\quad \text{ and }\quad  1-V_1(\xi_+)\leq C_+e^{-\lambda_1(c_1-c_0)t},\quad\text{for all}\ (t,x)\in \Omega_0^+(T),
\end{equation}
where 
$$
\Omega_0^+(T):=\Omega_0(T)\cap (T,\infty)\times [0,\infty).
$$
Moreover, up to enlarging $T>0$ if  necessary, there exists $0<\rho<\frac 1 3$ such that 
\begin{equation}
\label{rho-bis}
0<1-V_1(\xi_\pm)\leq \rho, \quad  \text{for all}\ (t,x)\in \Omega_0^+(T).
\end{equation}

\medskip

We first assume $d_c\le 1$.  By some straightforward computations and \eqref{def-s}, one has
\begin{equation*}\label{super sol N1}
N_1[\overline{u},\underline{v}]=\overline{u}\left(\frac{1-d_c}{2}t^{-1}+\frac{\tau e^{-\tau t}}{1-e^{-\tau t}}-2+V_1(\xi_+)+V_1(\xi_-)\right).
\end{equation*}
In view of \eqref{estimate of 1-V super sol}, by choosing $\tau<\lambda_1(c_1-c_0)$, we deduce that, for $T>0$ large enough, $N_1[\overline{u},\underline{v}]\geq 0$ in $\Omega_0^{+}(T)$. 

On the other hand, some straightforward computations yield
$$
N_2[\overline{u},\underline{v}]=J_1+J_2+J_3,
$$
where 
\begin{eqnarray*}
J_1&:=&-\overline{u}\left(\frac{\tau e^{-\tau t}}{1-e^{-\tau t}}+\frac{1-d_c}{2}t^{-1}+(1-d_c)\frac{\alpha_t}{\alpha}\right),\\
J_2&:=&2(1-V_1(\xi_-))(1-V_1(\xi_+)),\\
J_3&:=&\overline{u}(2-V_1(\xi_+)-V_1(\xi_-)).
\end{eqnarray*}
From the \lq\lq Heat kernel expression'' of $s(t,x)$, namely
$$
\alpha(t,x)=\left(G(t,\cdot)*\alpha_0\right)(x),\quad\text{where}\ G(t,x):=\frac{1}{\sqrt{4\pi t}}e^{-\frac{x^2}{4t}},
$$
we can check that $\alpha_t(t,x)\geq -\frac{1}{2t}\alpha(t,x)$. As a result, since $d_c\leq 1$, we have
\begin{equation*}
J_1\le -\overline{u}\frac{\tau e^{-\tau t}}{1-e^{-\tau t}}.
\end{equation*}

In view of (\ref{estimate of 1-V super sol}) and $\tau<\lambda_1(c_1-c_0)$, we have  $|J_3|=o(|J_1|)$ as $t\to\infty$.
Recalling $\alpha_0(x)=B_1e^{-q\vert x\vert}$, we have
$$
\alpha(t,x)=\frac{B_1}{\sqrt{4\pi t}}\left(\int _{-\infty}^0 e^{-\frac{(x-y)^2}{4t}}e^{qy}dy+\int_0^{+\infty}e^{-\frac{(x-y)^2}{4t}}e^{-qy}dy\right),
$$
which can be recast, after some elementary computations, 
\begin{equation}\label{expression of s}
\alpha(t,x)=\frac{B_1}{\sqrt \pi}\left(e^{q^2t-qx}\int_{\frac{2qt-x}{2\sqrt{t}}}^{+\infty}e^{-w^2}dw+e^{q^2t+qx}\int_{\frac{2qt+x}{2\sqrt{t}}}^{+\infty}e^{-w^2}dw\right).
\end{equation}
In particular, by setting  $2q<c_0$, we have
$$
\overline u(t,c_1t)\geq  \frac 12 \alpha(t,c_1t)\geq \frac{B_1}{4}e^{-(qc_0-q^{2})t}.
$$
Then, by setting $q$ small such that $qc_0-q^{2}<\lambda_1c_1$, we have $|J_2|=o(|J_1|)$ as $t\to\infty$.
As a result, $N_2[\overline{u},\underline{v}]\leq 0$ in $\Omega_0^{+}(T)$ up to enlarging $T$ if necessary.

\medskip
Next, we assume $d_c\ge 1$.
Some straightforward computations implies
\begin{equation*}\label{super sol N1 d>1}
N_1[\overline{u},\underline{v}]=\overline u\left(\frac{d_c-1}{2d_c}t^{-1}+\frac{d_c-1}{d_c}\frac{\alpha_t}{\alpha}+\frac{\tau e^{-\tau t}}{1-e^{-\tau t}}-2+V_1(\xi_+)+V_1(\xi_-)\right).
\end{equation*}
As above, since $d_c\ge 1$, $\alpha_t(t,x)\geq -\frac{1}{2t}\alpha(t,x)$ implies 
$$N_1[\overline{u},\underline{v}]\ge \overline{u}\left(\frac{\tau e^{-\tau t}}{1-e^{-\tau t}}-2+V_1(\xi_+)+V_1(\xi_-)\right).$$
In view of \eqref{estimate of 1-V super sol} and $\tau<\lambda_1(c_1-c_0)$,  we deduce that, for $T>0$ large enough, $N_1[\overline{u},\underline{v}]\geq 0$ in $\Omega_1^{+}(T)$. 

On the other hand, By some straightforward computations, one has
$$
N_2[\overline{u},\underline{v}]=J_1+J_2,
$$
where 
\begin{eqnarray*}
J_1&:=&\overline{u}\left(2-V_1(\xi_+)-V_1(\xi_-)-\frac{\tau e^{-\tau t}}{1-e^{-\tau t}}-\frac{d_c-1}{2d_c}t^{-1}\right),\\
J_2&:=&2(1-V_1(\xi_-))(1-V_1(\xi_+)).
\end{eqnarray*}
By applying the same argument as that for $d_c\le 1$, we get $N_2[\overline{u},\underline{v}]\leq 0$ in $\Omega_0^{+}(T)$.
\end{proof}
 

\subsubsection{Proof of Theorem \ref{th:profile c_c>c_f} for the high conversion rate case}
\begin{proposition}\label{prop: boundary condition super sol} 
Let $(F,C,H)$ be the solution of \eqref{fch-equation} with initial data satisfying \eqref{initial data}.
Let $(\overline{u},\underline{v})$ be given by \eqref{def-s}-\eqref{definition of super sol system} if $d_c\leq 1$, and by \eqref{def-s*}-\eqref{definition of super sol system d>1} if $d_c\geq 1$. Then there exist $0<q<\min\{\frac{c_0}{2d_c},\frac{c_0}{2},\frac{\lambda_1c_1}{c_0}\}$, $T^{**}>0$ and $B_1>0$  such that
$$
F(t,x)\leq \overline u(t,x)\quad \text{ and }\quad  \underline v(t,x)\le C(t,x), \quad \text{ for all }\; t\geq T^{**}, \vert x\vert \leq c_0t.
$$
\end{proposition}

\begin{proof} {\it of Proposition \ref{prop: boundary condition super sol}.}
We first fix $0<q<\min\{\frac{c_0}{2d_c},\frac{c_0}{2},\frac{\lambda_1c_1}{c_0}\}$ small enough so that
\begin{equation}
\label{choice-q}
\max\{qc_0-q^2, qc_0-d_cq^2\}<\min\{\sqrt{a}(c_0-c_f),\gamma_c\}.
\end{equation}
From Proposition \ref{prop:inequality super solution}, for any $t\geq T^{*}$, we have a super-solution $(\overline{u},\underline{v})$ for which $B_1>0$ is arbitrary. We only deal with the case $d_c\leq 1$, and the case $d_c\geq 1$ can be proved by the similar argument.

We first focus on $x =c_0t$ with $c_0<c_1$ and $t\geq T^*$ (the case $x=-c_0t$, $t\geq T^*$ being similar). 
\begin{equation}\label{cl:u< U on ct}
F(t,c_0t)\le \overline{u}(t,c_0t),\quad \text{for all}\ t\geq T^*,
\end{equation}
follows from \eqref{choice-q} and the upper estimate on $F$ obtained in section 2, and
\begin{equation}\label{cl:u< U v>V on ct}
\underline v(t,c_0t)\leq C(t,c_0t), \quad \text{for all}\ t\geq T^*,
\end{equation}
follows from \eqref{choice-q} and the estimate \eqref{C>1-e^-t} in Proposition \ref{prop: lower estimate on C a<1+s}.

Now, $T^*>0$ are fixed from the above discussion. We focus on the initial order, namely $t=T^{*}$, $\vert x\vert \leq c_0 T^{*}$. As above, we deduce from (\ref{expression of s}) that
$$
 \inf_{|x|\le c_0T^*}\overline{u}(T^*,x)\ge \frac 12  \alpha(T^*,c_0T^*)\geq \frac{B_1}{4}e^{-(qc_0-q^{2})T^*}\geq 1\geq \sup_{t>0, x\in \mathbb{R}} F(t,x),
 $$ 
 provided that $B_1\geq 4e^{(qc_0-q^2)T^*}$. On the other hand, one has
$$
\sup_{|x|\le c_0T^*}\underline{v}(T^*,x)\leq  1- \inf_{|x|\le c_0T^*}\overline{u}(T^*,x)\leq 1-\frac{B_1}{4}e^{-(qc_0-q^{2})T^*}\leq 0 \leq \inf _{t>0, x\in \mathbb{R}} F(t,x),
$$
for some large $B_1$ chosen by the same way.

It is easy to check that $(F,C)$ satisfying the equation \eqref{fch-equation} is the sub-solution of \eqref{critical competition system}. 
Then, by applying the comparison principle on $\Omega_1(T^{*})$, we conclude the proof of  Proposition \ref{prop: boundary condition super sol}.
\end{proof}

Combing with the upper estimate of $F$ in section 2 and Proposition \ref{prop: boundary condition super sol}, we find $F$ converges to $0$ uniformly for all $x\in\mathbb{R}$, and $C$ is uniformly greater than or equal to $1$ in the final zone.
\begin{proposition}\label{pr:a not= 1+s F}
Assume $d_c(1+s)>a$. Let $(F,C,H)$ be the solution of \eqref{fch-equation} with initial data satisfying \eqref{initial data}. Then one has
\begin{equation}\label{unif 0 F}
\limsup_{t\rightarrow\infty}\sup_{x\in{\mathbb R}}F(t,x)=0,
\end{equation}
\begin{equation}\label{C>=1 a<1+s}
\liminf_{t\rightarrow\infty}\inf_{|x|\le c_0t}C(t,x)\ge 1,\quad\text{for any}\quad 0<c_0<c_c.
\end{equation}
\end{proposition}

\begin{remark}
We would like to remark that  Proposition \ref{pr:a not= 1+s F} holds for both high conversion rate case $g\ge 1$
and low conversion rate case $g<1$. 
\end{remark}

Now, we are ready to prove the first statement of Theorem \ref{th:profile c_c>c_f}. 

\begin{proof}{\it of (1) of Theorem \ref{th:profile c_c>c_f}.} 
Assume $g\ge 1$. By working on $|x|\le c_0t$ for any $0<c_0<c_c$, from \eqref{C>=1 a<1+s} and the same limit argument in the proof of Claim \ref{cl:aa}, we obtain 
\begin{equation}\label{H=1 a<1+s g>1}
\limsup_{t\rightarrow\infty}\sup_{|x|\le c_0t}H(t,x)= 0.
\end{equation}
Then, by considering the limit system \eqref{limit system C H} with $H_{\infty}(t,x)\equiv 0$, one has 
\begin{equation}\label{C<1 a<1+s g>1}
\limsup_{t\rightarrow\infty}\sup_{|x|\le c_0t}C(t,x)\le 1.
\end{equation}
As a result, \eqref{asympt C H g>1 a<1+s} follows from \eqref{C>=1 a<1+s}, \eqref{H=1 a<1+s g>1}, and \eqref{C<1 a<1+s g>1}.
\end{proof}

\subsection{The bump phenomena in the high conversion case}
In this subsection, we will provide a lower estimate on $F$ and an upper estimate on $C$ for the case $g>1$, which explain the reason why bumps on the profiles of $F$ and $C$ are observed by the numerical simulation. A short discussion on the critical case $g=1$ will be given in Remark \ref{rm: g=1}.

The estimate \eqref{H=1 a<1+s g>1} implies that $H$ decreases to $0$ uniformly in the final zone. However, to get sharp estimates on $F$ and $C$, we need first show that $H$ decreases to $0$ with an exponential decay rate 
\begin{proposition}\label{prop: H go to 0 in finial zone}
Assume $d_c(1+s)>a$ and $g>1$. Let $(F,C,H)$ be the solution of \eqref{fch-equation} with initial data satisfying \eqref{initial data}. Then, for any $0<c_0<c_c$, there exist  $\tau_h>0$, $\tau_c>0$ and some large $T>0$ such that
\begin{equation}\label{H to 0 and C to 1 exponentially}
\sup_{|x|\le c_0t} H(t,x)\le e^{-\tau_h t}\ \text{and}\ \sup_{|x|\le c_0t} C(t,x)\le 1+e^{-\tau_c t}, \quad\text{for all}\quad t\ge T.
\end{equation}
\end{proposition}
\begin{proof}{\it of Proposition \ref{prop: H go to 0 in finial zone}.}
By \eqref{F unif to 0 for all x} and \eqref{asympt C H g>1 a<1+s}, we can assert that, there exist $\varepsilon>0$ and $T_0>0$ such that
for $t\ge T_0$, one has 
\begin{equation}\label{aaa}
\sup_{|x|\le c_0t,} H(t,x)\le \varepsilon\quad\text{and}\quad \sup_{|x|\le c_0t} |1-F(t,x)-C(t,x)|\le \varepsilon.
\end{equation}
From the $H$-equation in \eqref{fch-equation} and \eqref{aaa}, $H(t.x)$ satisfies
\begin{equation}\label{tilde H inequality}
\partial_t H\le d_h\partial_{xx} H-b H(g-1)(1-\varepsilon).
\end{equation}
Then, by applying the similar approach as that in the proof of Lemma 2.6 in \cite{PWZ}, we get
\begin{equation}\label{H to 0 exponentially}
\sup_{|x|\le c_0t} H(t,x)\le e^{-\tau_h t}, \quad\text{for all}\quad t\ge T_1.
\end{equation}

On the other hand, from \eqref{unif 0 F} and \eqref{aaa}, for $t\ge T_1$, $\tilde C:=C-1$ satisfies 
$$\sup_{|x|\le c_0t}|\tilde C(t,x)|\le \varepsilon.$$
Thus, it holds
$$\partial_t \tilde C\le d_c\partial_{xx}\tilde C-\tilde C(1-\varepsilon)+se^{-\gamma_ht}(1+\varepsilon).$$
Then, by applying the similar approach as that in the proof of Lemma 2.6 in \cite{PWZ}, we
can finish the proof of Proposition \ref{prop: H go to 0 in finial zone}.
\end{proof}

\begin{remark}\label{rm: g=1}
Note that, in the proof of Proposition \ref{prop: H go to 0 in finial zone}, it is necessary to assume $g>1$. For the critical case $g=1$, we believe there still exists a bump on the profile of $H$. Recall $(\overline u,\underline v)$ defined as \eqref{definition of super sol system} and \eqref{definition of super sol system d>1}. To show the bump phenomena on $H$, we need to use the solution of
$$\partial_t H= d_h\partial_{xx} H+b H(1-H-\underline v)$$
to construct super-solution for $H$, and use the solution of 
$$\partial_t H= d_h\partial_{xx} H-b H(H+\overline u),$$
to construct sub-solution for $H$ on $|x|\le c_0t$ with $0<c_0<c_c$. Since the computation is similar as what have been done in Proposition \ref{prop:inequality super solution}, we will not provide more details in the present paper. 
\end{remark}

\subsubsection{Construction of the sub-solution}
Let us fix $c_0<c_c<c_1$. We will construct a sub-solution on $\Omega_0(T)$ defined as \eqref{def:omega-un}  for the following system:
\begin{equation}\label{critical competition system sub solution}
\left\{
\begin{aligned}
&\partial_t\tilde u=\partial_{xx} \tilde u+a\tilde u(1-\tilde u-\tilde v),\\
&\partial_t\tilde v=d_c\partial_{xx}\tilde v+\tilde v(1-\tilde v-\tilde u)+se^{-\gamma_ht}(\tilde u+\tilde v),\\
&\tilde v(t,x)\ge se^{-\gamma_ht},\ (t,x)\in \Omega_0(T).
\end{aligned}
\right.
\end{equation}
Thus, in this subsection, we define the operators
\begin{equation}\label{def of sup and sub sol}
N_1[\tilde u,\tilde v]:=\partial_t\tilde u-\partial_{xx} u-\tilde u(1-\tilde u-\tilde v)\quad \text{ and } \quad  N_2[\tilde u,\tilde v]:=\partial_t\tilde v-d_c\partial_{xx}\tilde v-\tilde v(1-\tilde v-\tilde u)-se^{-\gamma_ht}(\tilde u+\tilde v).
\end{equation}
Note that, the system \eqref{critical competition system sub solution} is a competition system since the condition
\begin{equation}\label{condition of competition system}
\tilde v(t,x)\ge se^{-\gamma_ht}\quad\text{for all}\quad(t,x)\in \Omega_0(T).
\end{equation}
Recall that $V_{1}$ is a traveling wave solution with speed $c_1$ defined as \eqref{single-V1}.
The construction of the sub-solution is more involved than the super-solution.

We introduce $\alpha_1(t,x)$ and $\alpha_2(t,x)$ the solutions of the Cauchy problems
\begin{equation}
\label{def-f-h}
\left\{
\begin{aligned}
\partial_t\alpha_1&= \partial_{xx}\alpha_1,\\
\alpha_1(0,x)&=B_2\mathbf 1 _{(-1,1)}(x),
\end{aligned}
\right.
\qquad 
\left\{
\begin{aligned}
\partial_t\alpha_2&= \partial_{xx}\alpha_2,\\
\alpha_2(0,x)&=B_3e^{-k|x|},
\end{aligned}
\right.
\end{equation}
and consider a sub-solution $(\underline u,\overline v)$ for which all parameters are positive:
\begin{equation}\label{definition of sub sol system}
\left\{
\begin{aligned}
\underline u(t,x)&:=\beta(t)\alpha_1(t,x)-\alpha_2(t,x),\\
\overline v(t,x)&:=V_{1}(x-c_1t-\zeta_0)+V_{1}(-x-c_1t-\zeta_0)-1-\underline u(t,x)+\frac{B_4}{t^{1+\theta}},
\end{aligned}
\right.
\end{equation}
where
$$
\beta(t):=\exp \frac{B_4}{\delta(1+t)^{\delta}},
$$
and $V_1$ is the traveling wave solution with speed $c_1>c_c$ defined as \eqref{single-V1}.

\begin{proposition}\label{prop:sub}  Let $0<\delta<\theta <\frac 12$ and $k>0$  be given. Let us set $B_3=\gamma B_2$ with $0<\gamma<1$.
Then there exists $T^{*}>0$ such that, for all $0<B_2<1$, $B_4>1$, and $\zeta _0>0$, $(\underline u,\overline v)$ is a sub-solution in $\Omega_0(T^{*})$.
\end{proposition}

\begin{proof}{\it of Proposition \ref{prop:sub}.} Since $\underline u(t,\cdot)$ and $\overline v(t,\cdot)$ are even functions, it is enough to work for $t\geq T$ and $0\leq x<c_0t$.  For simplicity of notations, we shall use the shortcuts  $\xi_{\pm}:=\pm x-c_1t-\zeta_0$. Since $\xi_-\leq -c_1t$ and $\xi_+\leq -(c_v^{**}-c_2)t$, it follows from $V_1'<0$ and \eqref{estimate of V at infinity-un} that  there exist $C_->0$ and $C_+>0$ such that, for $T>0$ large enough,
\begin{equation}\label{estimate of 1-V}
1-V_1(\xi_-)\leq C_-e^{-\lambda_1c_1t} \quad \text{ and }\quad  1-V_1(\xi_+)\leq C_+e^{-\lambda_1(c_1-c_0)t},\quad  \text{for all}\ (t,x)\in \Omega_0^+(T),
\end{equation}
where $\Omega_0^+(T)=\Omega_0(T)\cap (T,\infty)\times [0,\infty)$. Moreover, up to enlarging $T>0$ if necessary, there exists $0<\rho<\frac 1 3$ such that 
\begin{equation}
\label{rho}
0<1-V_1(\xi_\pm)\leq \rho, \quad  \text{for all}\ (t,x)\in \Omega_0^+(T).
\end{equation}

By some straightforward computations and  \eqref{def-f-h}, one has
\begin{eqnarray*}
N_1[\underline u,\overline v]&=&\beta'\alpha_1-(\beta\alpha_1-\alpha_2)(2-V_1(\xi_+)-V_1(\xi_-)-B_4t^{-(1+\theta)})\\
&\leq &-B_4(1+t)^{-(1+\delta)}\beta\alpha_1+B_4\beta\alpha_1t^{-(1+\theta)}+\alpha_2\left(2-V_1(\xi_+)-V_1(\xi_-)-B_4t^{-(1+\theta)}\right),
\end{eqnarray*}
since $2-V_1(\xi_+)-V_1(\xi_-)>0$.  Thus, it follows from \eqref{estimate of 1-V} that
$$
N_1[\underline u,\overline v]\leq B_4\beta\alpha_1\left(  -(1+t)^{-(1+\delta)}+t^{-(1+\theta)}\right)+\alpha_2\left(C_-e^{-\lambda_1c_1t}+C_+e^{-\lambda_1(c_1-c_0)t}-B_4t^{-(1+\theta)}\right).
$$
Since  $\delta<\theta$, it follows that, for $T>0$ large enough,  $N_1[\underline u,\overline v]\le 0$ in $\Omega_0^{+}(T)$.

On the other hand, by some straightforward computations and \eqref{def-f-h}, one has 
\begin{eqnarray*}
N_2[\underline u,\overline v]:=I_1+\cdots +I_6,
\end{eqnarray*}
where
\begin{eqnarray*}
I_1&:=&\beta\alpha_1\left(2-V_1(\xi_+)-V_1(\xi_-)+ B_4(1+t)^{-(1+\delta)}\right),\\
I_2&:=&(1-V_1(\xi_-))\left(2-2V_2(\xi_+)-\alpha_2+se^{-\gamma_ht}\right),\\
I_3&:=&-(1-V_1(\xi_+))\alpha_2,\\
I_4&:=&(d_c-1)(\beta \partial_t\alpha_1- \partial_t\alpha_2),\\
I_5&:=&B_4t^{-(1+\theta)}\left(2V_1(\xi_-)+2V_1(\xi_+)-3+B_4t^{-(1+\theta)}-\beta\alpha_1+\alpha_2-(1+\theta)t^{-1}\right),\\
I_6&:=&-se^{-\gamma_ht}V_1(\xi_+).
\end{eqnarray*}
Since $0<V_1<1$, we have $I_1\geq 0$. From \eqref{estimate of 1-V}, there exists $C_1,C_2>0$ such that
$$
I_2\geq -(1-V_1(\xi_-))\alpha_2\ge -C_1t^{-\frac{1}{2}}e^{-\lambda_1c_1t},
$$
and $I_3\geq -C_2t^{-\frac{1}{2}}e^{-\lambda_1(c_1-c_0)t}$. Last, from \eqref{rho} and the estimates in Lemma 5.3 in \cite{AX}, there exits $C_3,C_4>0$ such that
$$
I_4+I_5\ge B_4t^{-(1+\theta)}\left(1-3\rho-\Vert \beta\Vert _\infty C_3 t^{-\frac 12}-(1+\theta)t^{-1}\right)-C_3\vert d_c-1\vert (\Vert \beta\Vert _\infty +1)t^{-\frac 32}\ge C_4B_4t^{-(1+\theta)},
$$ 
since $\theta<\frac 12$ and $0<\rho <\frac 13$. As a result, for $T^*>0$ large enough, $N_2[\underline u,\overline v]\geq 0$ in $\Omega_0^{+}(T^*)$. Moreover, by enlarging $T^*$ if necessary, \eqref{condition of competition system} holds for all $(t,x)\in \Omega_0(T^*)$. 
Note that, time $T^*$ in Proposition \ref{prop:sub} is independent on $0<B_2<1$, $B_4>1$, and $\zeta _0>0$. Therefore, we can reduce $B_2$ and to enlarge $B_4$ and $\zeta _0$ such that \eqref{condition of competition system} holds for all $(t,x)\in\Omega_0(T^*)$.
The proof of Proposition \ref{prop:sub} is complete.
\end{proof}

\subsubsection{Proof of Theorem \ref{th: bump}}

 To apply the comparison principle, we also need the suitable order on the boundary of the domain, which will be obtained by choosing a suitable $k$ and the estimate from Lemma 5.1 in \cite{AX}.

\begin{proposition}\label{prop:first}
Assume $d_c(1+s)>a$ and $g>1$. Let $(F,C,H)$ be the solution of \eqref{fch-equation} with initial data satisfying \eqref{initial data}.
Let $0<\delta<\theta <\frac 12$ be given. Let us fix $k:=\frac{c_0}{2}>0$, and set $B_3=\gamma B_2$ with $0<\gamma<1$. 
Then there exist $T^{**}>0$, $0<B_2<1$, $B_4>1$, and $\zeta _0>0$ such that
$$
\underline u(t,x)\leq F(t,x)\quad \text{ and }\quad  C(t,x)\le \overline v(t,x), \quad \text{ for all }\; t\geq T^{**}, \vert x\vert \leq c_0t,
$$
where $(\underline u,\overline v)$ is given by \eqref{definition of sub sol system}.
\end{proposition}

\begin{proof}{\it of Proposition \ref{prop:first}.} We aim at applying the comparison principle in $\Omega_0(T)$ with a suitable $T>0$. From Proposition \ref{prop:sub}, for any $t\geq T^{*}$, we have a sub-solution $(\underline u,\overline v)$ for which $0<B_2<1$, $B_4>1$, and $\zeta_0>0$ are arbitrarily chosen. 
On the other hand, by the estimate on $H$ from \eqref{H to 0 exponentially} and the lower estimate on $C$ from Proposition \ref{prop: boundary condition super sol}, the solution $(F,C)$ of \eqref{fch-equation} is a super-solution of \eqref{critical competition system sub solution} for $(t,x)\in\Omega_0(T^*)$ up to enlarging $T^*$ if necessary.

We now focus on  $\vert x\vert =c_0t$, $t\geq T^*$. 
Recall the choice $k=\frac{c_2}{2}$ which follows from Lemma 5.1 in \cite{AX} that
$$
\underline u(t,\pm c_0 t) \leq 0\leq F(t,\pm c_0 t), \quad \text{for all}\ t\geq T^{*},
$$
up to enlarging $T^*$ if necessary.

Next, it follows from  Proposition \ref{prop: H go to 0 in finial zone} that $C(t,\pm c_0t)\le 1+ e^{-\gamma_ct}$ as $t\to\infty$. On the other hand, for any $t\geq T^{**}$,
$$
\overline v(t,\pm c_0 t)\geq 1+\frac{1}{2t^{1+\theta}}.
$$
As a result, up to enlarging $T^{**}$ if  necessary, one has
$$
C(t,\pm c_0 t)\leq \overline v(t\pm c_0 t), \quad \text{for all } t\geq T^{**}.
$$

Last we focus on the order of the initial data, namely $t=T^{**}$ and $\vert x\vert \leq c_0 T^{**}$. From Proposition \ref{prop: H go to 0 in finial zone}, there exists $\varepsilon_1,\varepsilon_2>0$ such that
$$
\inf_{|x|\le c_0T^{**}}F(T^{**},x)\geq \varepsilon_1,\sup_{|x|\le c_0T^{**}}C(T^{**},x)-1\leq \varepsilon_2.
$$
We now set 
\begin{equation}\label{B2 B4}
0<B_2 <\frac{\varepsilon_1}{2}e^{-\frac{B_4}{\delta(1+T^{**})^{\delta}}},
\end{equation}
 which implies
$$
\underline u(T^{**},x)\leq B_2 \beta(T^{**}) \leq \frac{\varepsilon_1}{2} \leq F(T^{**},x),\quad \text{ for all } \vert x\vert \leq c_0 T^{**}.
$$
And we chose $B_4>1$ and $\zeta_0>0$  large enough such that
\begin{equation*}\label{11}
\overline v(T^{**},x)\ge 2V_{1}(-\zeta_0)-1-\frac{\varepsilon_1}{2}+B_4(T^{**})^{-(1+\theta)}\geq 1+\varepsilon_2,
\end{equation*}
which implies
$$
C(T^{**},x)\leq \overline v(T^{**},x), \quad\text{for all}\quad \vert x\vert \leq c_0 T^{**}.
$$

As a consequence, the comparison principle can be applied in $\Omega_2(T^{**})$, which concludes the proof of Proposition \ref{prop:first}.
\end{proof}

Theorem \ref{th: bump} follows immediately from Proposition \ref{prop: boundary condition super sol} and \ref{prop:first}, and some classical estimates on  solutions of the heat equation. We refer to Lemma 5.1 in \cite{AX} for more details.

\subsection{Asymptotic profiles for the low conversion rate case}
In this subsection, we deal with the proof of statement (2) in  Proposition \ref{th:profile c_c>c_f}. The key point of studying the asymptotic profiles in the low conversion rate case is to provide a uniform lower estimate of $H$ in the final zone. In our previous study \cite{MX}, for the general case, we thought it is hard to give the necessary and sufficient condition under which $H$ is uniform positive from below.
In this subsection, we first show that $H$ is uniformly positive from below for all $(t,x)\in\mathbb{R}^+\times\mathbb{R}$ if and only if $g<1$. Then, we show that $C$ and $H$ would converge to $(C^*,H^*)$ as $t\to+\infty$ in the final zone.

\subsubsection{A uniformly lower estimate on $H$ for the low conversion rate case}
In this subsection, we aim to prove the following proposition.
\begin{proposition}\label{prop: uniform lower estimate on H g<1}
Assume $g<1$ and $d_c(1+s)>0$. Let $(F,C,H)$ be the solution of \eqref{fch-equation} with initial data satisfying \eqref{initial data}. Then there exists $\varepsilon_h>0$ such that
$$\underset{t\to+\infty}{\liminf}\underset{x\in\mathbb{R}}{\inf}H(t,x)\ge \varepsilon_h.$$
\end{proposition}

Before stating our arguments, we would like to introduce some basic properties at first.
Denote $X$ as a Banach space of $\mathbb{R}^3$-valued bounded and uniformly continuous functions on $\mathbb{R}$ endowed with the usual sup-norm. Let us define $\Psi\subset X$ as
$$\Psi=\{(\psi_1,\psi_2,\psi_3)\in X:\ \psi_1\ge 0,\ \psi_2\ge 0,\ \mbox{and}\ 0\le\psi_3\le 1\ \mbox{for all}\ x\in\mathbb{R}\}.$$
Let us denote the  nonlinear semiflow generated by the system (\ref{fch-equation}) by $Z(t)$,
then one may obtain that, it holds
$$Z(t)[\Psi]\subset \Psi\ \ \mbox{for all}\ \  t>0.$$

First we prove a lemma  which plays a key role in the proof of Proposition \ref{prop: uniform lower estimate on H g<1}. In particular, it holds for any $(F_0,C_0,H_0)\in \Psi$ satisfying $H_0\not\equiv 0$.  
\begin{lemma}\label{low estimate of H 3}
Let $(F,C,H)$ be the solution of \eqref{fch-equation} with initial data $(F_0,C_0,H_0)\in \Psi$ satisfying $H_0\not\equiv 0$. Then 
there exists $\varepsilon_1>0$ such that
\begin{equation*}
\underset{t\to+\infty}{\limsup}\ H(t,0)\ge \varepsilon_1.\\
\end{equation*}
\end{lemma}
\begin{proof}\it{of Lemma \ref{low estimate of H 3}.}
\normalfont
For $F_0(x)+C_0(x)\equiv 0$, the lemma holds immediately with $\varepsilon_1=1.$ Hence, without loss of generality, we assume $F_0(x)+C_0(x)\not\equiv 0$.
We argue by contradiction by assuming there exist sequences
$$\{(F_{0,n},C_{0,n},H_{0,n})\}_{n\ge 0}\subset\Psi\ \ {\rm with}\ \  H_{0,n}\not\equiv0, \ \ \mbox{and}\ \ \{t_n\}_{n\ge0}\subset[0,\infty)\ \ \mbox{satisfying}\ \ t_n\to+\infty,$$
such that the following statement holds
\begin{equation}\label{construction of H (lower estimate) c<c*}
H_n(t,0)\le \frac{1}{n}\ \ \mbox{for all}\ \ t\ge t_n,
\end{equation}
wherein $(F_n,C_n,H_n)$ denotes the solution with the initial data $(F_{0,n},C_{0,n},H_{0,n})$. 
Then, one can claim that 
\begin{claim}\label{claim 6}
If the statement \eqref{construction of H (lower estimate) c<c*} holds true, then for any sequence $\{t'_n\}_{n\ge0}$ satisfying $t'_n\ge t_n$ and $R>0$, it holds:
\begin{equation}\label{uniform F to 0 in claim}
\lim_{n\to\infty}\sup_{t\ge0,x\in B_R}F_n(t'_n+t,x)= 0,
\end{equation}
\begin{equation}\label{uniform C<1 in claim}
\lim_{n\to\infty}\sup_{t\ge0,x\in B_R}C_n(t'_n+t,x)\le 1,
\end{equation}
\begin{equation}\label{uniform converget to 0 of H in claim}
\lim_{n\to\infty}\sup_{t\ge0,x\in B_R}H_n(t'_n+t,x)=0.
\end{equation}
\end{claim}

\begin{proof}\it{of Claim \ref{claim 6}.}
\normalfont
The statement \eqref{uniform F to 0 in claim} follows from the upper estimate on $F$ in \eqref{unif 0 F}.
Then we prove that \eqref{construction of H (lower estimate) c<c*} implies the statements \eqref{uniform C<1 in claim} and \eqref{uniform converget to 0 of H in claim} hold true.
To proceed by contradiction, we assume that there exists $R>0$, there exist $\delta>0$, $s_n>t_n$ and $x_n\in B_R$ such that for $n\ge N$,
$$H_n(s_n,x_n)\ge \delta.$$
Due to standard parabolic estimates, possibly along a subsequence, one may assume that 
\begin{equation*}
\left\{
\begin{aligned}
&\underset{n\to\infty}{\lim}F_n(s_n+t,x_n+x)= F_{\infty}(t,x),\\
&\underset{n\to\infty}{\lim}C_n(s_n+t,x_n+x)= C_{\infty}(t,x),\\
&\underset{n\to\infty}{\lim}H_n(s_n+t,x_n+x)= H_{\infty}(t,x).
\end{aligned}
\right.
\end{equation*}
The above convergences hold locally uniformly in $(t,x)\in\mathbb{R}\times \mathbb{R}^N$ and $(F_{\infty},C_{\infty},H_{\infty})$ is an entire solution of the following system
\begin{equation}\label{F+C infty system}
\left\{
\begin{aligned}
&\partial_tF_{\infty}=\partial_{xx} F_{\infty}+aF_{\infty}(1-C_{\infty}-F_{\infty}),\\
&\partial_tC_{\infty}=d_c\partial_{xx} C_{\infty}+C_{\infty}(1-C_{\infty}-F_{\infty})+s(F_{\infty}+C_{\infty})H_{\infty},\\
&\partial_tH_{\infty}=d_h\partial_{xx} H_{\infty}+bH_{\infty}(1-gF_{\infty}-gC_{\infty}-H_{\infty}).
\end{aligned}
\right.
\end{equation}

From the strong maximum principle, the construction \eqref{construction of H (lower estimate) c<c*}, and \eqref{uniform F to 0 in claim} , one has $H_{\infty}\equiv 0$ and $F_{\infty}\equiv 0$. However, since the sequence $\{x_n\}\subset B_R$ is relatively compact, we may assume that $\underset{n\rightarrow\infty}{\lim}\,x_n=x_{\infty}\in\widebar{B_R}$ by extracting sub sequence and hence $H_{\infty}(0,x_{\infty})\ge \delta$. This contradicts $H_{\infty}\equiv 0$ and yields \eqref{uniform converget to 0 of H in claim}.

Next, we assume that there exists $R'>0$, there exist $\delta'>0$, $s'_n>t_n$ and $x'_n\in B_R$ such that for $n\ge N$,
$$C_n(s_n,x_n)\ge 1+\delta.$$
By considering the limit system again, we get 
$$\partial_tC_{\infty}=d_c\partial_{xx} C_{\infty}+C_{\infty}(1-C_{\infty}),$$
which implies $C_{\infty}\le 1$. However, since the sequence $\{x'_n\}\subset B_R$ is relatively compact, we may assume that $\underset{n\rightarrow\infty}{\lim}\,x'_n=x'_{\infty}\in\widebar{B_R}$ by extracting sub sequence and hence $C_{\infty}(0,x'_{\infty})\ge 1+\delta'$. This contradicts $C_{\infty}\le 1$ and proves that the statement \eqref{uniform C<1 in claim} holds true.  
\end{proof}
\vspace{10pt}

Now, we can go back to the proof of Lemma \ref{low estimate of H 3}.  From \eqref{uniform F to 0 in claim} and \eqref{uniform C<1 in claim}, for any $R>0$ and small enough $\delta>0$ , for any $n$ large enough, one has
\[
(F_n+C_n)(t_n+t,x)\le (1+s+\delta)\chi_{\mathbb{R}^N\setminus B_R}+(1+\delta)\chi_{B_R}(x):=\widebar{G}(x)
\]
for all $t>0$ and $x\in\mathbb{R}$,
Then, by the comparison principle, we obtain
\[
H_n(t_n+t,x)\ge \wideubar{H}_n(t,x)\ \ \mbox{for all}\ \ t\ge 0,\ x\in\mathbb{R},
\]
wherein $\wideubar{H}_n$ is the solution of the equation
\begin{equation}\label{equation of sub solution of H_n g<1}
\left\{
\begin{array}{rl}
&\partial_t\wideubar{H}_n=d_h\partial_{xx}\wideubar{H}_n+b\wideubar{H}_n(1-\wideubar{H}_n-g\widebar{G}),\\
&\wideubar{H}_n(0,x)=H_n(t_n,x).
\end{array}
\right.
\end{equation}

Then, we consider a stationary sub-solution $\psi(0,x;\eta)$ for each $\eta>0$ as
\[
\psi(0,x;\eta)=\eta\phi_R(x),
\]
where $\phi_R(x)$ is the eigenfunction defined as \eqref{eigenfunction}. 
One can check that there exists $\eta_0>0$ such that for any $\delta$ small enough, $0<\eta\le \eta_0$ and $R$ large enough, the function $\psi(0,x)$ is a stationary sub-solution of the equation (\ref{equation of sub solution of H_n g<1}).
Therefore, the solution $\psi(t,x;\eta)$ of \eqref{equation of sub solution of H_n g<1} from initial data $\psi(0,x;\eta)$ is increasing in time, and converges to some positive stationary solution that denote by $p_{n,R,\delta}(x)$. 
Moreover, the stationary state
$p_{n,R,\delta}(x)$ does not depend on the choice of $\eta\in(0,\eta_0]$\ . 
\begin{claim}\label{claim  of stationary solution}
The stationary solution $p_{n,R,\delta}(x)$ does not depend on the choice of $\eta\in(0,\eta_0]$\ . 
\end{claim}
\begin{proof}\it{of Claim \ref{claim  of stationary solution}.}
\normalfont
To check this, let us change our notation for simplicity and denote the stationary sub-solution and  this stationary solution as $\psi$ and $p_{\eta}$. 
We first note that the comparison principle implies that $p_{\eta}\le p_{\eta'}$ for any $\eta <\eta'$. Next, let us assume by contradiction that there exists $\eta_1<\eta_0$ with $p_{\eta_1}\not\equiv p_{\eta_0}$. Hence, inferring from the strong maximum principle, one has $p_{\eta_1}< p_{\eta_0}$. Moreover, there exists a point $x_0\in B_R$ such that $\psi(0,x_0;\eta_0)>p_{\eta_1}(x_0)$. If not, then $\psi(0,x;\eta_0)\le p_{\eta_1}(x)$ for $x\in\mathbb{R}$, which implies $p_{\eta_1}\ge p_{\eta_0}$, and obtains a contradiction.

Then, we consider
\[
\eta^*=\sup\{\eta\in[\eta_1,\eta_0]\ ;\ \psi(0,x;\eta)\le p_{\eta_1}(x)\ \ \mbox{for all}\ \  x\in\mathbb{R}\}.
\]
One can deduce from the comparison principle and the strong maximum principle that 
\begin{equation}\label{xxx}
\psi(0,x;\eta^*)<\psi(t,x;\eta^*)<p_{\eta_1}(x)\ \ \mbox{for all}\ \ t>0,\ x\in\mathbb{R}. 
\end{equation}

On the other hand, from the definition of $\eta^*$, since $\psi$ has compact support $B_R$, there exists $x_0\in B_R$ such that $\psi(0,x_0;\eta^*)=p_{\eta_1}(x_0)$, which reaches a contradiction. 
\end{proof} 
\vspace{10pt}

Now, we can choose $\eta$ sufficiently small such that $\wideubar{H}_n(0,x)\ge\psi(0,x;\eta)$ for all $x\in\mathbb{R}$. Then, it follows from the comparison principle that for any
$R>0$ large enough and $\delta>0$ small enough and $n$ large enough
\begin{equation}\label{eq 8}
\liminf_{t\to\infty}H_n(t_n+t,x)\ge \liminf_{t\to\infty}\wideubar{H}_n(t,x)\ge p_{n,R,\delta}(x)\ \ \mbox{for all}\ \ x\in\mathbb{R}.
\end{equation}

To complete the proof of Lemma \ref{low estimate of H 3}, it remains to check that $p_{n,R,\delta}$ is far way from $0$ as $n\to\infty$, $R\to\infty$ and $\delta\to 0$.
Since $p_{n,R,\delta}$ is bounded from above by $1$, one can use standard elliptic estimates to get that, as $n\to +\infty$, $R\to+\infty$ and $\delta\to 0$, the function $p_{n,R,\delta}(x)$ converges locally uniformly to a stationary solution $p_{\infty}(x)$ of the equation
\begin{equation*}
d_h\partial_{xx} p_{\infty}+bp_{\infty}(1-g-p_{\infty})=0.
\end{equation*}
Moreover, since the map $t\to\psi(t,x;\eta_0)$ is nondecreasing, we assert from Claim \ref{claim  of stationary solution} that $p_{n,R,\delta}(0)\ge\psi(0,0;\eta_0)\ge\eta_0\phi_R(0)$. Note that $\phi_R(x)\to 1$ locally uniformly as $R\to+\infty$, and hence $p_{\infty}(0)\ge\eta_0$. As a consequence, we reached a contradiction to \eqref{uniform converget to 0 of H in claim} and completed the proof of Lemma \ref{low estimate of H 3}.
\end{proof}
\vspace{10pt}

Now we are ready to prove Proposition \ref{prop: uniform lower estimate on H g<1}.

\begin{proof}\it{of Proposition \ref{prop: uniform lower estimate on H g<1}.}
\normalfont
We proceed by contradiction and  assume that there exists a sequences $\{t_n\}_{n>0}$ with $t_n\to\infty$ and $\{x_n\}_{n>0}\subset\mathbb{R}$ such that
\begin{equation*}
\liminf_{t\to+\infty}H(t_n,x_n)\le \frac{1}{n}.
\end{equation*}

The initial data $H_0(x)\equiv 1$ implies that there exist two sequences $\{s_n\}_{n\ge 0}$ with $0<s_n<t_n$ such that for each $n>0$,

\begin{equation*}
H(t,x_n)\le \frac{\varepsilon_1}{2}\ \ \mbox{for all}\ \ t\in[s_n,t_n],
\end{equation*}
\begin{equation*}
H(s_n,x_n)=\frac{\varepsilon_1}{2}.
\end{equation*}

We assume as before, possibly along a subsequence, the functions 
\[
(F_n, C_n, H_n)(t,x):=(F,C,H)(t_n+t,x_n+x)
\]
converges locally uniformly to $(F_{\infty}, C_{\infty}, H_{\infty})$, which is an entire solution of
\begin{equation}\label{eq F C H infty without drift 2}
\left\{
\begin{aligned}
&\partial_tF_{\infty}=\partial_{xx} F_{\infty}+aF_{\infty}(1-C_{\infty}-F_{\infty}),\\
&\partial_tC_{\infty}=d_c\partial_{xx} C_{\infty}+C_{\infty}(1-C_{\infty}-F_{\infty})+s(F_{\infty}+C_{\infty})H_{\infty},\\
&\partial_tH_{\infty}=d_h\partial_{xx} H_{\infty}+bH_{\infty}(1-gF_{\infty}-gC_{\infty}-H_{\infty}).
\end{aligned}
\right.
\end{equation}
From the choices of sequences $\{t_n\}_{n\ge 0}$ and $\{x_n\}_{n\ge 0}$, one has $H_{\infty}(0,0)=0$, and hence $H_{\infty}\equiv 0$. In particular, $|s_n-t_n|\to\infty$ as $n\to\infty$, otherwise it contradicts the fact that
\begin{equation*}
\lim_{n\to\infty}H_n(s_n-t_n,0)=\frac{\varepsilon_1}{2}>0.
\end{equation*}

Now let us consider the limit functions as follows:
\begin{equation*}
\widetilde{F}(t,x)=\lim_{n\to\infty}F_n(t,x)=\lim_{n\to\infty}F(s_n+t,x_n+x),
\end{equation*}
\begin{equation*}
\widetilde{C}(t,x)=\lim_{n\to\infty}C_n(t,x)=\lim_{n\to\infty}C(s_n+t,x_n+x),
\end{equation*}
\begin{equation*}
\widetilde{H}(t,x)=\lim_{n\to\infty}H_n(t,x)=\lim_{n\to\infty}H(s_n+t,x_n+x),
\end{equation*}
which are well defined thanks to standard parabolic estimates. 
Then we look on $(\widetilde{F}, \widetilde{C}, \widetilde{H})$ as a solution of (\ref{fch-equation}) with initial data
\begin{equation*}
(\widetilde{F}_0, \widetilde{C}_0, \widetilde{H}_0):=\lim_{n\to\infty}(F_n(s_n,x_n+x), C_n(s_n,x_n+x), H_n(s_n,x_n+x)).
\end{equation*}

Since $\widetilde{H}_0(0)=\varepsilon_1/2>0$, by applying Lemma \ref{low estimate of H 3}, one has
\begin{equation}\label{eq 5}
\limsup_{t\to\infty}\widetilde{H}(t,0)\ge \varepsilon_1. 
\end{equation}
On the other hand, for all $t\in[0,t_n-s_n)$, it holds
\begin{equation*}
H_n(t,0)\le \frac{\varepsilon_1}{2}.
\end{equation*}
For each fixed $t\ge 0$, since $t\in[0,t_n-s_n)$, we get by the locally uniform convergence that
\begin{equation*}
\widetilde{H}(t,0)\le \frac{\varepsilon_1}{2} \ \ \mbox{for all}\ \  t\ge 0,
\end{equation*}
which contradicts the result \eqref{eq 5}.
Therefore, the proof of Proposition \ref{prop: uniform lower estimate on H g<1} is complete. 
\end{proof}
\vspace{10pt}

\subsubsection{Construction of the strict Lyapunov function}

Before we investigate the profiles of the $C$-component and $H$-component in the finial zone,  we first consider the dynamics of the underlying ODE system \eqref{Ode system}:
\[
\left\{
\begin{aligned}
&C_t=C(1-C)+sCH,\\
&H_t=bH(1-H-gC).
\end{aligned}
\right.
\]
We expect the solution of the PDE system (\ref{fch-equation}) to converge uniformly to the unique positive equilibrium $(C^*,H^*)$ as $t\to+\infty$.

Let us introduce the set $\Sigma=\{(C,H)\in \mathbb{R}^2\ :\ 0<C<1+s,\ 0<H<1 \}$.
There exists
a strictly convex function $\Phi:\Sigma\to\mathbb{R}$ of class $C^2$ that attains its minimum point at $(C^*,H^*)$ and satisfies
\[
(C(1-C+sH),bH(1-H-gC))\cdot \nabla\Phi(C,H)\le 0\ \mbox{for all}\ \ (C,H)\in\Sigma.
\]
As a matter of fact, we can consider the strictly convex functional as
$$\Phi(C,H):=bg\int_{C^*}^C\frac{\eta-C^*}{\eta}d\eta+s\int_{H^*}^H\frac{\xi-H^*}{\xi}d\xi.$$
It is not difficult to check that, for all $(C,H)\in \Sigma$, it holds
$$(C(1-C+sH),bH(1-H-gH))\cdot \nabla\Phi(C,H)=-bg(C-C^*)^2-bs(H-H^*)^2\le 0.$$

Moreover, we claim that there exists $\nu=\min\{C^*,bH^*\}>0$ such that 
\begin{equation}\label{condition on lp function}
(C(1-C+sH),bH(1-H-gH))\cdot \nabla\Phi(C,H)\le -\nu\Phi(C,H).
\end{equation}
Indeed, for $C\le C^*$, since $\int_{C^*}^C\frac{\eta-C^*}{\eta}d\eta\le 0$, one has
$$-(C-C^*)^2\le -\nu \int_{C^*}^C\frac{\eta-C^*}{\eta}d\eta.$$
For $C> C^*$, since $\int_{C^*}^C\frac{\eta-C^*}{\eta}d\eta\ge \frac{(C-C^*)^2}{C^*}$,
it holds 
$$-bg(C-C^*)^2\le - bgC^*\int_{C^*}^C\frac{\eta-C^*}{\eta}d\eta.$$
By some similar computations, one also has
$$-bs(H-H^*)^2\le -\nu s\int_{H^*}^H\frac{\xi-H^*}{\xi}d\xi.$$

Furthermore, for any solution $(C,H)$ of the ODE system \eqref{Ode system}, one has
\begin{align*}
\Phi(C,H)_t&=bg(C-C^*)(1-C+sH)+bs(H-H^*)(1-H-gH)\\
&=-bg(C-C^*)^2-bs(H-H^*)^2,
\end{align*}
which implies it is a strict Lyapunov function in the sense that:
if $(C,H)$ denotes the solution of corresponding ODE system \eqref{Ode system} with 
the initial data $(C_0,H_0)$, then
\[
\Phi(C(t),H(t))=\Phi(C_0,H_0)\ \mbox{for all}\ \ t>0\Rightarrow (C_0,H_0)=(C^*,H^*).
\]
Since $\Phi$ is bounded from below, we assume without loss of generality that $\Phi\ge 0$, and the equality only holds at the unique
minimizer $(C^*,H^*)$.

\subsubsection{Proof of Theorem \ref{th:profile c_c>c_f} for the low conversion rate case}
We will first show that, in the final zone, $C$ is uniformly smaller than $1+s$, and $H$ is uniformly smaller than $1+s$.
\begin{proposition}\label{prop: C<1+s H<1}
Assume $a<d_c(1+s)$ and $g<1$. Let $(F,C,H)$ be the solution of \eqref{fch-equation} with initial data satisfying \eqref{initial data}. Then for any $0<c_0<c_c$, there exists $\varepsilon_2$ such that
\begin{equation}\label{C<1+s}
\limsup_{t\rightarrow\infty}\sup_{|x|\le c_0t}C(t,x)\le 1+s-\varepsilon_2,
\end{equation}
\begin{equation}\label{H<1}
\limsup_{t\rightarrow\infty}\sup_{|x|\le c_0t}H(t,x)\le 1-\varepsilon_2.
\end{equation}
\end{proposition}
\begin{proof}{\it of Proposition \ref{prop: C<1+s H<1}.}
For arbitrarily chosen $\{t_n\}_{n>0}$ with $t_n\to\infty$ and $\{x_n\}_{n>0}$ satisfying $|x_n|<(c_0-\varepsilon) t_n$, we define 
\[
(F_n, C_n, H_n)(t,x):=(F,C,H)(t_n+t,x_n+x),
\]
which converges locally uniformly to $(F_{\infty}, C_{\infty}, H_{\infty})$ satisfying the limit system \eqref{eq F C H infty without drift 2}.
By \eqref{unif 0 F}, $(C_{\infty},H_{\infty})$ satisfies
\begin{equation*}
\left\{
\begin{aligned}
&\partial_tC_{\infty}=d_c\partial_{xx} C_{\infty}+C_{\infty}(1-C_{\infty})+sC_{\infty}H_{\infty},\\
&\partial_tH_{\infty}=d_h\partial_{xx} H_{\infty}+bH_{\infty}(1-gC_{\infty}-H_{\infty}).
\end{aligned}
\right.
\end{equation*}

By \eqref{C>=1 a<1+s}, we can assert that $C_{\infty}(t,x)\ge 1$ for all $(t,x)\in\mathbb{R}\times\mathbb{R}$. Then
$$\overline H_{\infty}(t,x):=(1-\frac{g}{2})+\frac{g}{2}e^{-\frac{g(t+t_0)}{2}}$$
is the super-solution, which implies $H_{\infty}(t,x)\le 1-\frac{g}{2}$ for all $(t,x)\in\mathbb{R}\times\mathbb{R}$.

On the other hand, since $H_{\infty}(t,x)\le 1-\frac{g}{2}$ for all $(t,x)\in\mathbb{R}\times\mathbb{R}$, it is easy to check that 
$$\overline C_{\infty}(t,x):=(1+s(1-\frac{g}{4}))+\frac{g}{4}e^{-\delta(t+t_0)}$$ 
with very small $\delta>0$ is the super-solution, which implies $C_{\infty}(t,x)\le 1+s(1-\frac{g}{4})$ for all $(t,x)\in\mathbb{R}\times\mathbb{R}$.
Since $\{t_n\}_{n>0}$ and $\{x_n\}_{n>0}$ are arbitrarily chosen, the proof of Proposition \ref{prop: C<1+s H<1} is complete.
\end{proof}

We are  now ready to complete the proof of Theorem \ref{th:profile c_c>c_f}.

\begin{proof}{\it of (2) of Theorem \ref{th:profile c_c>c_f}.}
Let us argue by contradiction and assume that there exist $c\in\hspace{-3pt}[0,c^*)$ and a sequence $\{(t_n,x_n)\}_{n\ge0}\in\mathbb{R}^+\times\mathbb{R}$ such that $t_n\to+\infty$ and $\delta>0$ such that for all $n>0$,
\begin{equation}\label{aaa}
|x_n|\le ct_n\ \ \mbox{and}\ \ |C(t_n,x_n)-C^*|+|H(t_n,x_n)-H^*|\ge\delta.
\end{equation}
Consider the sequence of functions $(F_n,C_n,H_n)(t,x)=(F,C,H)(t+t_n,x+x_n)$.
Now, let us fix $c'>0$ such that $c<c'<c^*$. Therefore, there exist $N>0$ large enough and $\varepsilon>0$ small enough such that, for  $t+t_n\ge A$ and $|x|\le c't+(c'-c)t_n$, one has $F_n(t,x)\le \frac{1}{n}$ from Proposition \ref{pr:a not= 1+s F}. Moreover, there exists $\varepsilon>0$ such that 
$$\varepsilon\le C_n(t,x)\le 1+s-\varepsilon\quad\text{and}\quad \varepsilon\le H_n(t,x)\le 1-\varepsilon.$$

Then, by parabolic estimates, possibly along a subsequence, one may assume that
$$(F_n,C_n,H_n)(t,x)\to (F_{\infty},C_{\infty},H_{\infty})(t,x)\ \ \mbox{locally uniformly for}\ \ (t,x)\in\mathbb{R}\times\mathbb{R},$$
where $(F_{\infty},C_{\infty},H_{\infty})$ is a bounded entire solution and satisfies
\[
\underset{(t,x)\in\mathbb{R}\times\mathbb{R}}{\sup}F_{\infty}(t,x)=0,
\]
\begin{equation}\label{C>0 H>0 unif}
\underset{(t,x)\in\mathbb{R}\times\mathbb{R}}{\inf}C_{\infty}(t,x)>0\ \ \mbox{and}\ \ \underset{(t,x)\in\mathbb{R}\times\mathbb{R}}{\inf}H_{\infty}(t,x)>0,
\end{equation}
Moreover,  by \eqref{aaa} it holds
\begin{equation}\label{bb}
|C_{\infty}(0,0)-C^*|+|H_{\infty}(0,0)-H^*|>0.
\end{equation}
However, by applying the Lyapunov function $\Psi(C,H)$ which satisfying \eqref{condition on lp function},  \eqref{bb} contradicts to the convergence result of Theorem 1.1 in \cite{GS}. 
Therefore, the proof of (2) of Theorem \ref{th:profile c_c>c_f} is complete.
\end{proof}


\section*{Acknowledgments}
The authors would like to express sincere thanks to Prof. Hiroshi Matano for many helpful suggestions and  continuous encouragement.


\end{document}